\documentclass{amsart}
\usepackage[left=1in,right=1in,top=1in,bottom=1in]{geometry}
 \usepackage{microtype}
\usepackage{hyperref}
         
\usepackage{amsmath}             
\usepackage{amsfonts}             
\usepackage{amsthm}               
\usepackage{amsbsy}
\usepackage{amssymb}
\usepackage{amsthm}
\usepackage{amsrefs}

\newtheorem{thm}{Theorem}[subsection]
\newtheorem{lem}[thm]{Lemma}
\newtheorem{prop}[thm]{Proposition}
\newtheorem{cor}[thm]{Corollary}

\theoremstyle{definition}
\newtheorem{defn}[thm]{Definition}
\newtheorem{nota}[thm]{Notation}
\newtheorem{rem}[thm]{Remark}
\newtheorem{exam}[thm]{Example}
\newtheorem{cons}[thm]{Construction}

\newcommand{\bC}{{\mathbb{C}}}
\newcommand{\bN}{{\mathbb{N}}}

\newcommand{\A}{{\mathcal{A}}}
\newcommand{\B}{{\mathcal{B}}}
\newcommand{\E}{{\mathcal{E}}}

\newcommand{\I}{{\mathcal{I}}}

\renewcommand{\L}{{\mathcal{L}}}

\renewcommand{\P}{{\mathcal{P}}}
\renewcommand{\S}{{\mathcal{S}}}

\newcommand{\X}{{\mathcal{X}}}
\newcommand{\Y}{{\mathcal{Y}}}

\renewcommand{\phi}{\varphi}

\newcommand{\qand}{\quad\text{and}\quad}
\newcommand{\qqand}{\qquad\text{and}\qquad}

\newcommand{\AND}{\text{ and }}
\newcommand{\alg}{\mathrm{alg}}

\usepackage{tikz}
\usetikzlibrary{shapes,snakes,calc,arrows}
\usetikzlibrary{decorations.pathreplacing,shapes.geometric}
\usetikzlibrary{calc,positioning}
\tikzset{Box/.style={very thick, rounded corners}}
\tikzset{marked/.style={star, star point height = .75mm, star points =5, fill=black,minimum size=2mm, inner sep=0mm} }
\tikzset{verythickline/.style = {line width=7pt}}
\tikzset{thickline/.style = {line width=5pt}}
\tikzset{medthick/.style = {line width=3pt}}
\tikzset{med/.style = {line width=2pt}}
\tikzset{count/.style = {fill=white,circle,draw,thin, inner sep=2pt}}
\tikzset{rcount/.style = {fill=white,rectangle,draw,thin,inner sep=2pt, rounded corners}}
\tikzset{cpr/.style = {draw,fill=white,rectangle,thin, rounded corners}}

\newcommand{\bncb}{{BNC_{b}}}
\newcommand{\bncm}{{BNC_{m}}}
\newcommand{\bncs}{{BNC_{vs}}}

\begin{document}

\nocite{*}

\title[Independences and Partial $R$-Transforms in Bi-Free Probability]{Independences and Partial $R$-Transforms \\ in Bi-Free Probability}

\author{Paul Skoufranis}
\address{Department of Mathematics, Texas A\&M University, College Station, Texas, USA, 77843}
\email{pskoufra@math.tamu.edu}

\date{\today}

\begin{abstract}
In this paper, we examine how various notions of independence in non-commutative probability theory arise in bi-free probability.  
We exhibit how Boolean and monotone independence occur from bi-free pairs of faces and establish a Kac/Loeve Theorem for bi-free independence.
In addition, we prove that bi-freeness is preserved under tensoring with matrices.
Finally, via combinatorial arguments, we construct partial $R$-transforms in two settings relating the moments and cumulants of a left-right pair of operators.
\end{abstract}

\maketitle

\section{Introduction}

In non-commutative probability theory there are several notions of independence which characterize the joint moments of collections of algebras in terms of the individual algebras.  By \cite{S1997} there are precisely three symmetric notions of independence: classical independence, free independence, and Boolean independence.  In addition, there are two antisymmetric notions of independence: monotone and anti-monotone independence.  These notions of independence have very similar theories such as the existence of cumulants and the ability to use power series transformations describing moments of convolutions of operators.  We refer the reader to \cites{V1985, SW1997, HS2011, M2001, S1994, B2005, L2004, P2008, P2009} for the development of these theories.

In \cite{V2013-1} Voiculescu introduced the notion of bi-free independence in order to simultaneously study the left and right reduced representations of algebras on reduced free product spaces.  Instead of being an independence for collections of algebras, which would reduce to one of the above five notions by \cites{M2002, M2003}, bi-free independence is an independence for pairs of algebras; one designated the left algebra and the other designated the right algebra.  In comparison with the other notions of independence, the cumulants for bi-free independence, known as the $(\ell, r)$-cumulants, were developed in \cites{MN2013, CNS2014-1, CNS2014-2} and a partial $R$-transform for a left-right pair of operators was discussed in \cite{V2013-2}.

Voiculescu noticed both classical and free independence can be viewed as specific instances of bi-free independence.  In particular, given two algebras $A_1$ and $A_2$ in a non-commutative probability space $(\A, \varphi)$, $A_1$ and $A_2$ are classically independent if and only if $(A_1, \mathbb{C}1_\A)$ and $(\mathbb{C} 1_\A, A_2)$ are bi-freely independent, and $A_1$ and $A_2$ are freely independent if and only if $(A_1, \mathbb{C} 1_\A)$ and $(A_2, \bC 1_\A)$ are bi-freely independent.  Furthermore, given a collection of bi-freely independent pairs of algebras, each left algebra is classically independent from every right algebra from a different pair, the left algebras are freely independent, and the right algebras are freely independent.

This paper investigates how other notions of independence occur inside bi-free probability, how bi-free independence shares many properties with free independence, and partial $R$-transforms for left-right pairs of operators.  One can either view these results as indications that classical, free, Boolean, and monotone independence are specific instances of bi-free independence, or as information on how bi-free independence works in specific instances in an effort to reduce the mystery surrounding bi-free probability.  There are seven sections to this paper, including this introduction, which are structured and summarized as follows.

Section \ref{sec:Background} reviews the notation, structures, and results from \cite{CNS2014-2} necessary to discuss bi-free independence with amalgamation.  In particular, the notions of bi-non-crossing partitions and diagrams, the bi-non-crossing M\"{o}bius function, $B$-$B$-non-commutative probability spaces, operator-valued bi-multiplicative functions, and bi-free families of pairs of $B$-faces are recalled.

Section \ref{sec:ClassicalAndFreeIndependence} examines how classical and free independence are specific instances of bi-free independence.  In particular, it is demonstrated how the moments of classical/free independent algebras occur by summing over certain subsets of bi-non-crossing partitions.  Furthermore, a bi-free Kac/Loeve Theorem is developed thereby demonstrating that any two bi-free pairs of algebras that remain bi-free after a non-trivial rotation must be bi-free central limit distributions.

Section \ref{sec:Boolean} examines how (operator-valued) Boolean independence arises inside bi-free probability.  To begin, Theorem \ref{thm:algebras-are-Boolean-independent} demonstrates how Boolean independent algebras arise from bi-free pairs of algebras under mild moment hypotheses.  Furthermore, given a collection $\{A_k\}_{k \in K}$ of Boolean independent algebras, we demonstrate two different constructions in order to obtain a family of bi-free pairs of algebras such that $\{A_k\}_{k \in K}$  embed into Boolean independent algebras generated by the bi-free pairs of algebras.  Only one of these embeddings is a homomorphism, but the non-multiplicative embedding allows one, without knowledge of \cite{SW1997}, to define the Boolean cumulants as specific $(\ell, r)$-cumulants.  In addition, it is shown how the moments of Boolean independent algebras occur by summing certain $(\ell, r)$-cumulants corresponding to bi-non-crossing partitions that resemble interval partitions.

Section \ref{sec:Monotone} examines how (operator-valued) (anti)monotone independence arises inside bi-free probability.  As in Section \ref{sec:Boolean}, we demonstrate how monotonically independent algebras arise from bi-free pairs of algebras under mild moment hypotheses and how every pair of monotonically independent algebras can be embedded into the bi-free setting.  In summary, classical, free, Boolean, and monotone independence can be realized as specific instances of bi-free independence and the moment functions of these independences are given by summing over specific bi-non-crossing partitions as roughly described below.
\begin{center}
\begin{tabular}{| l | l | }
\hline
\textbf{Independence} & \textbf{Bi-Non-Crossing Partitions Used}\\ \hline 
Free & partitions that have only left (or right) nodes \\ \hline
Classical & partitions that have both left and right nodes but \\
& no block contains both  \\ \hline
Boolean &  partitions that have both left and right nodes and \\
&  every block contains both \\ \hline
Monotone & partitions that have both left and right nodes but\\ 
& no block contains both and no block can connect \\
& two left nodes if it needs to pass a right node \\ \hline
\end{tabular}
\end{center}

Section \ref{sec:Tensoring} examines matrices of bi-free pairs of algebras.   Several advances in free probability, such as those of \cite{BMS2013}, revolve around the ability to use matrices of operators to simplify the computations for the moments of the operators.  Essential to this is the fact that matrices of freely independent algebras are free with amalgamation over $M_n(\bC)$ with respect to the amplified state.  Theorem \ref{thm:bi-free-preserved-when-tensored-with-Mn} demonstrates the same holds in the bi-free setting; matrices of bi-freely independent algebras are bi-free with amalgamation over $M_n(\bC)$.

Section \ref{sec:Transforms} uses the combinatorics of bi-free probability to examine partial $R$-transforms.  
The Cauchy transform and $R$-transform, which have played an essential role in free probability, were first examined by Voiculescu in \cite{V1986}. 
Subsequently Speicher in \cite{S1994} used combinatorics to derive the relation between the Cauchy transform and the $R$-transform.
Furthermore, Speicher and Woroudi used similar methods in \cite{SW1997} to derive expressions for Boolean independence.
In Section \ref{sec:Transforms} a simple, purely combinatorial proof of the partial $R$-transform for a left-right pair of operators constructed in \cite{V2013-2}*{Theorem 2.4} is given.  Finally, using similar techniques, another partial $R$-transform is constructed whose proof generalizes the proof for the Boolean transforms from \cite{SW1997}*{Proposition 2.1}.

\section{Background on Bi-Freeness with Amalgamation}
\label{sec:Background}

This section reviews the background and notation for bi-freeness with amalgamation required in the remainder of the paper.  We refer the reader to \cites{CNS2014-1, CNS2014-2} for more details.

In general, a map $\chi : \{1, \ldots, n\} \to \{\ell, r\}$ is used to designate whether each operator from a set of $n$ operators should be a left or a right operator and a map $\epsilon : \{1,\ldots, n\} \to K$ is used to determine which algebra from a collection of algebras indexed by $K$ each operator is from.

\subsection{Bi-Non-Crossing Partitions}  
Let $\P (n)$ denote the set of partitions on $n$ elements.  Given two partitions $\pi, \sigma \in \P (n)$, we say that $\pi$ is a refinement of $\sigma$, denoted $\pi \leq \sigma$, if every block of $\pi$ (a set in $\pi$) is contained in a single block of $\sigma$.  Refinement defines a partial ordering on $\P (n)$ turning $\P (n)$ into a lattice.

Given $\chi: \{1, \ldots, n\} \to \{\ell, r\}$, if 
\[
\chi^{-1}(\{\ell\}) = \{i_1<\cdots<i_p\} \qqand \chi^{-1}(\{r\}) = \{i_{p+1} > \cdots > i_n\},
\]
define the permutation $s_\chi$ on $\{1,\ldots, n\}$ by $s_\chi(k) = i_k$.  In addition, define the total ordering $\prec_\chi$ on $\{1,\ldots, n\}$ by $a\prec_\chi b$ if and only if $s_\chi^{-1} (a)< s_\chi^{-1}(b)$.  Notice $\prec_\chi$ corresponds to, instead of reading $\{1,\ldots, n\}$ in the traditional order, reading $\chi^{-1}(\{\ell\})$ in increasing order followed by reading $\chi^{-1}(\{r\})$ in decreasing order.  

A subset $V \subseteq \{1,\ldots, n\}$ is said to be a \emph{$\chi$-interval} if $V$ is an interval with respect to the ordering $\prec_\chi$.
In addition, $\min_{\prec_\chi}(V)$ and $\max_{\prec_\chi}(V)$ denote the minimal and maximal elements of $V$ with respect to the ordering $\prec_\chi$.

\begin{defn}
A partition $\pi \in \P (n)$ is said to be \emph{bi-non-crossing with respect to $\chi$} if the partition  $s_\chi^{-1}\cdot \pi$ (the partition formed by applying $s_\chi^{-1}$ to the blocks of $\pi$) is non-crossing.
Equivalently $\pi$ is bi-non-crossing if whenever there are blocks $U, V \in \pi$ with $u_1, u_2 \in U$ and $v_1, v_2 \in V$ such that $u_1 \prec_\chi v_1 \prec_\chi u_2 \prec_\chi v_2$, then $U = V$.  The set of bi-non-crossing partitions with respect to $\chi$ is denoted by $BNC(\chi)$.
\end{defn}

Note $BNC(\chi)$ inherits a lattice structure from $\P (n)$ and thus has minimal and maximal elements, denoted $0_\chi$ and $1_\chi$ respectively.

To each partition $\pi \in BNC(\chi)$ we associate a \emph{bi-non-crossing diagram} as follows:  place nodes along two dashed vertical lines, labelled $1$ to $n$ from top to bottom, such that the nodes on the left line correspond to those values for which $\chi(k) = \ell$ and nodes on the right line correspond to those values for which $\chi(k) = r$.  Then use lines to connect the nodes which are in the same block of $\pi$ in such a way that lines from different blocks do not cross.

\begin{exam}
If $\chi^{-1}(\{\ell\})=\{1,2,4\}$, $\chi^{-1}(\{r\})=\{3,5\}$, and
	\begin{align*}
		\pi=\left\{\vphantom{\sum}\{1,3\}, \{2,4,5\} \right\}= s_\chi\cdot \left\{\vphantom{\sum} \{1,5\}, \{2,3,4\} \right\},
	\end{align*}
then the bi-non-crossing diagram associated to $\pi$ is
	
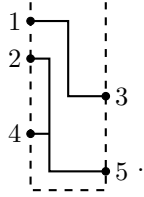
\begin{align*}
	\begin{tikzpicture}[baseline]
	\node[ left] at (-.5, 2) {1};
	\draw[fill=black] (-.5,2) circle (0.05);
	\node[left] at (-.5, 1.5) {2};
	\draw[fill=black] (-.5,1.5) circle (0.05);
	\node[right] at (.5, 1) {3};
	\draw[fill=black] (.5,1) circle (0.05);
	\node[left] at (-.5, .5) {4};
	\draw[fill=black] (-.5,.5) circle (0.05);
	\node[right] at (.5,0) {5};
	\draw[fill=black] (.5,0) circle (0.05);
	\draw[thick] (-.5,2) -- (0,2) -- (0, 1) -- (.5,1);
	\draw[thick] (-.5,1.5) -- (-.25,1.5) -- (-.25, 0) -- (.5,0);
	\draw[thick] (-.5,.5) -- (-.25,.5);
	\draw[thick, dashed] (-.5,2.25) -- (-.5,-.25) -- (.5,-.25) -- (.5,2.25);
	\end{tikzpicture}.
	\end{align*}
\end{exam}

In such diagrams, the vertical lines are referred to as spines.

\subsection{The Bi-Non-Crossing M\"{o}bius Function}  
The \emph{bi-non-crossing M\"{o}bius function} is the function
\[
\mu_{BNC} : \bigcup_{n\geq1}\bigcup_{\chi : \{1, \ldots, n\}\to\{\ell, r\}}BNC(\chi)\times BNC(\chi) \to \bC
\]
defined such that $\mu_{BNC}(\pi, \sigma) = 0$ unless $\pi$ is a refinement of $\sigma$, and otherwise defined recursively via the formulae
\[
\sum_{\substack{\tau \in BNC(\chi) \\\pi \leq \tau \leq \sigma}} \mu_{BNC}(\tau, \sigma) = \sum_{\substack{\tau \in BNC(\chi) \\ \pi \leq \tau \leq \sigma}} \mu_{BNC}(\pi, \tau) = \left\{
\begin{array}{ll}
1 & \mbox{if } \pi = \sigma  \\
0 & \mbox{otherwise}
\end{array} \right. .
\]
Due to the similarity of the lattice structures, the bi-non-crossing M\"{o}bius function is related to the non-crossing M\"{o}bius function $\mu_{NC}$ by the formula 
\[
\mu_{BNC}(\pi, \sigma) = \mu_{NC}(s^{-1}_\chi \cdot \pi, s^{-1}_\chi \cdot \sigma).
\]

\subsection{$\boldsymbol{B}$-$\boldsymbol{B}$-Non-Commutative Probability Space}  To discuss bi-freeness with amalgamation, the correct abstract structures are required.  For this section and the rest of the paper, $B$ denotes a unital algebra over $\bC$.  
\begin{defn}
A \emph{$B$-$B$-bimodule with a specified $B$-vector state} is a triple $(\X, \mathring{\X}, p)$ where $\X$ is a direct sum of $B$-$B$-bimodules
\[
\X = B \oplus \mathring{\X},
\]
and $p : \X \to B$ is the linear map
\[
p(b \oplus \eta) = b.
\]

Let $\L(\X)$ denote the set of linear operators on $\X$.  For each $b \in B$ define the operators $L_b, R_b \in \L(\X)$ by
\[
L_b(\eta) = b \cdot \eta \qquad \mbox{ and } \qquad R_b(\eta) = \eta \cdot b \qquad \text{ for all } \eta \in \X.
\]
The unital subalgebras of $\L(\X)$ defined by
\begin{align*}
\L_\ell(\X) &:= \{ Z \in \L(\X) \, \mid \, ZR_b = R_b Z \mbox{ for all }b \in B\} \text{ and}\\
\L_r(\X) &:= \{ Z \in \L(\X) \, \mid \, ZL_b = L_b Z \mbox{ for all }b \in B\}
\end{align*}
are called the \emph{left} and \emph{right algebras of} $\L(\X)$ respectively.
\end{defn}
It is important to note that $\L_\ell(\X)$ consists of all operators in $\L(\X)$ that are right $B$-linear and thus are potential operators for the left face of a pair of $B$-faces (see Definition \ref{defn:pair-of-B-faces}).
\begin{defn}
Given a $B$-$B$-bimodule with a specified $B$-vector state $(\X, \mathring{\X}, p)$, the \emph{expectation of $\L(\X)$ onto $B$} is the linear map $E_{\L(\X)} : \L(\X) \to B$ defined by
\[
E_{\L(\X)}(Z) = p(Z(1_B))
\]
for all $Z \in \L(\X)$.
\end{defn}
It was shown in \cite{CNS2014-2}*{Proposition 3.1.6} that $E_{\L(\X)}$ has two essential properties; namely
\[
E_{\L(\X)}(L_{b_1} R_{b_2} Z) = b_1 E_{\L(\X)}(Z) b_2
\]
for all $b_1, b_2 \in B$ and $Z \in \L(\X)$, and
\[
E_{\L(\X)}(ZL_b) = E_{\L(\X)}(ZR_b)
\]
for all $b \in B$ and $Z \in \L(\X)$.  Based on these properties, the following abstract structures were examined.
\begin{defn}
A \emph{$B$-$B$-non-commutative probability space} is a triple $(\mathcal{A}, E_\mathcal{A}, \varepsilon)$ where $\mathcal{A}$ is a unital algebra, $\varepsilon : B \otimes B^{\mathrm{op}} \to \mathcal{A}$ is a unital homomorphism such that $\varepsilon|_{B \otimes 1_B}$ and $\varepsilon|_{1_B \otimes B^{\mathrm{op}}}$ are injective, and $E_\mathcal{A} : \mathcal{A} \to B$ is a linear map such that
\[
E_{\mathcal{A}}(\varepsilon(b_1 \otimes b_2)Z) = b_1 E_{\mathcal{A}}(Z) b_2
\]
for all $b_1, b_2 \in B$ and $Z \in \mathcal{A}$, and
\[
E_{\mathcal{A}}(Z\varepsilon(b \otimes 1_B)) = E_{\mathcal{A}}(Z\varepsilon(1_B \otimes b))
\]
for all $b \in B$ and $Z \in \mathcal{A}$.

The unital subalgebras of $\A$ defined by
\begin{align*}
\mathcal{A}_\ell &:= \{ Z \in \mathcal{A}  \, \mid \, Z\varepsilon(1_B \otimes b) = \varepsilon(1_B \otimes b) Z \mbox{ for all }b \in B\} \text{ and}\\
\mathcal{A}_r &:= \{ Z \in \mathcal{A}  \, \mid \, Z\varepsilon(b \otimes 1_B) = \varepsilon(b \otimes 1_B) Z \mbox{ for all }b \in B\}
\end{align*}
are called the \emph{left} and \emph{right algebras of} $\A$ respectively.  To simplify notation, $L_b$ and $R_b$ are used in place of $\varepsilon(b \otimes 1_B)$ and $\varepsilon(1_B \otimes b)$ respectively. 
\end{defn}
In the case that $B = \bC$, one sees that $(\A, E, \varepsilon)$ is nothing more than a non-commutative probability space; that is, a pair $(\A, \varphi)$ where $\A$ is a unital algebra and $\varphi : \A \to \bC$ is a unital linear map.

It is useful to compare the notion of a $B$-$B$-non-commutative probability space with the notion of a $B$-non-commutative probability space used in free probability.
\begin{defn}
A \emph{$B$-non-commutative probability space} is a pair $(\A, \Phi)$ where $\A$ is a unital algebra containing $B$ (with $1_\A= 1_B$) and $\Phi : \A \to B$ is a unital linear map such that
\[
\Phi(b_1 Z b_2) = b_1 \Phi(Z) b_2
\]
for all $b_1, b_2 \in B$ and $Z \in \A$.
\end{defn}

\begin{rem}
\label{rem:B-ncps-from-B-B-ncps}
In free probability, one is interested in the joint $B$-moments
\[
\Phi(b_0 Z_1 b_1 Z_2 b_2 \cdots Z_n b_n) = b_0 \Phi(Z_1 b_1 Z_2 b_2 \cdots Z_n) b_n
\]
for $Z_1, \ldots, Z_n \in \A$ and $b_0, b_1, \ldots, b_n \in B$.  Such moments can be naturally recovered in the bi-free setting.  Indeed $\A$ is naturally a $B$-$B$-bimodule via left and right multiplication by $B$ and thus can be made into a $B$-$B$-bimodule with specified $B$-vector space via $p = \Phi$ and $\mathring{\X} = \ker(\Phi)$.  Hence $\L(\A)$ is a $B$-$B$-non-commutative probability state with
\[
E_{\L(\A)}(Z) = \Phi(Z 1_\A)
\]
for all $Z \in \L(\A)$.  

Notice $\A$ may be viewed as a unital subalgebra of both $\L_\ell(\A)$ and $\L_r(\A)$ by left and right multiplication on $\A$ respectively.  Viewing $\A \subseteq \L_\ell(\A)$, one can recover the joint $B$-moments of elements of $\A$ from $E_{\L(\A)}$ since
\[
E_{\L(\A)}(L_{b_0} Z_1 L_{b_1} Z_2 \cdots Z_n L_{b_n}) = \Phi(b_0 Z_1 b_1 Z_2 b_2 \cdots Z_n b_n) 
\]
for $Z_1, \ldots, Z_n \in \A$ and $b_0, b_1, \ldots, b_n \in B$.  

Furthermore, given a $B$-$B$-non-commutative probability space $(\A, E, \varepsilon)$, notice $(\A_\ell, E)$ is always a $B$-non-commutative probability space with $\varepsilon(B \otimes 1_B)$ as the copy of $B$.  Indeed
\[
E_\A(L_{b_1} Z L_{b_2}) = E_\A(L_{b_1} Z R_{b_2}) = E_\A(L_{b_1} R_{b_2} Z) = b_1 E_\A(Z) b_2
\]
for all $Z \in \A_\ell$ and $b_1, b_2 \in B$.  Furthermore, $(\A_r, E)$ is a $B^{\mathrm{op}}$-non-commutative probability space with $\varepsilon(1_B \otimes B^{\mathrm{op}})$ as the copy of $B^{\mathrm{op}}$.
\end{rem}

 The essential property of a $B$-$B$-non-commutative probability space is its ability to be concretely represented on a $B$-$B$-bimodule with a specified vector state.
\begin{thm}[\cite{CNS2014-2}*{Theorem 3.2.4}]
\label{thm:representing-bb-ncps}
Let $(\mathcal{A}, E_\mathcal{A}, \varepsilon)$ be a $B$-$B$-non-commutative probability space.  Then there exists a $B$-$B$-bimodule with a specified $B$-vector state $(\X, \mathring{\X}, p)$ and a unital homomorphism $\theta : \mathcal{A} \to \L(\X)$ such that $\theta(L_{b_1} R_{b_2}) = L_{b_1} R_{b_2}$,
\[
\theta(\mathcal{A}_\ell) \subseteq \L_\ell(\X), \quad
\theta(\mathcal{A}_r) \subseteq \L_r(\X), \quad
\mathrm{and} \quad
E_{\L(\X)}(\theta(Z)) = E_\mathcal{A}(Z)
\]
for all $b_1, b_2 \in B$ and $Z \in \mathcal{A}$.
\end{thm}

\subsection{Operator-Valued Bi-Multiplicative Functions}  For discussions on bi-freeness with amalgamation, one needs the correct notions for moment and cumulant functions and the properties these functions have.

Given $\chi: \{1,\ldots, n\} \to \{\ell, r\}$ and a subset $X \subseteq \{1,\ldots, n\}$, let $\chi|_X : X \to \{\ell, r\}$ denote the restriction of $\chi$ to $X$.
Similarly, given an $n$-tuple of objects $(Z_1, \ldots, Z_n)$, let $(Z_1, \ldots, Z_n)|_X$ denote the $|X|$-tuple where the elements in positions not indexed by an element of $X$ are removed.
Finally, given $\pi \in BNC(\chi)$ such that $X$ is a union of blocks of $\pi$, let $\pi|_X \in BNC(\chi|_X)$ denote the bi-non-crossing partition formed by taking the blocks of $\pi$ contained in $X$.

\begin{defn}
\label{defn:recursive-definition-of-E-pi}
Let $(\A, E, \varepsilon)$ be a $B$-$B$-non-commutative probability space.  For $\chi : \{1,\ldots,n\} \to \{\ell, r\}$, $\pi \in BNC(\chi)$, and $Z_1, \ldots, Z_n \in \A$, we define 
\[
E_\pi(Z_1,\ldots, Z_n) \in B
\]
recursively as follows.
Let $V$ be the block of $\pi$ that terminates closest to the bottom of the bi-non-crossing diagram associated to $\pi$. Then:
\begin{itemize}
\item If $V = \{1,\ldots, n\}$ (that is, $\pi = 1_\chi)$,
\[
E_{1_\chi}(Z_1, \ldots, Z_n) := E(Z_1 \cdots Z_n).
\]
\item If $\min(V)$ is not adjacent to any spines of $\pi$, then $V = \{k+1, \ldots, n\}$ for some $k \in \{1, \ldots, n-1\}$  and
\[
E_\pi(Z_1, \ldots, Z_n) := \left\{
\begin{array}{ll}
E_{\pi|_{V^c}}(Z_1, \ldots, Z_k L_{E_{\pi|_V}(Z_{k+1},\ldots, Z_n)}) & \mbox{if } \chi(\min(V)) = \ell  \\
E_{\pi|_{V^c}}(Z_1, \ldots, Z_k R_{E_{\pi|_V}(Z_{k+1},\ldots, Z_n)}) & \mbox{if } \chi(\min(V)) = r
\end{array} \right..
\]
\item Otherwise, $\min(V)$ is adjacent to a spine. Let $W$ denote the block of $\pi$ corresponding to the spine adjacent to $\min(V)$ and
let $k$ be the smallest element of $W$ that is larger than $\min(V)$.  If $\chi(\min(V)) = \ell$ we defined
\[
E_\pi(Z_1, \ldots, Z_n) := E_{\pi|_{V^c}}((Z_1, \ldots, Z_{k-1}, L_{E_{\pi|_V}((Z_{1},\ldots, Z_n)|_V)} Z_k, Z_{k+1}, \ldots, Z_n)|_{V^c})
\]
and if $\chi(\min(V)) = r$ we define
\[
E_\pi(Z_1, \ldots, Z_n) := E_{\pi|_{V^c}}((Z_1, \ldots, Z_{k-1}, R_{E_{\pi|_V}((Z_{1},\ldots, Z_n)|_V)} Z_k, Z_{k+1}, \ldots, Z_n)|_{V^c}).
\]
\end{itemize}
\end{defn}
For an example expression, see \cite{CNS2014-2}*{Example 5.1.2}.  Observe that, in the context of Definition \ref{defn:recursive-definition-of-E-pi}, we ignore the notions of left and right operators and do not specify whether each entry of $E_\pi$ is a left or right operator based on $\chi$.
However, we are interested in making this restriction.
\begin{defn}
Let $(\mathcal{A}, E, \varepsilon)$ be a $B$-$B$-non-commutative probability space.  The \emph{bi-free operator-valued moment function}
\[
\mathcal{E} : \bigcup_{n\geq 1} \bigcup_{\chi : \{1,\ldots, n\} \to \{\ell, r\}} BNC(\chi) \times \mathcal{A}_{\chi(1)} \times \cdots \times \mathcal{A}_{\chi(n)} \to B
\]
is defined by
\[
\mathcal{E}_\pi(Z_1, \ldots, Z_n) = E_\pi(Z_1, \ldots, Z_n)
\]
for each $\chi : \{1, \ldots, n\} \to \{\ell, r\}$, $\pi \in BNC(\chi)$, and $Z_k \in \mathcal{A}_{\chi(k)}$.
\end{defn}

In the case that $B = \bC$, recall $E : \A \to \bC$ is a linear map denoted by $\varphi$.  In this case, if $Z_1, \ldots, Z_n \in A$ and $\pi \in BNC(\chi)$ has blocks $V_t = \{k_{t, 1}<\cdots<k_{t,m_t}\}$ for $t \in \{1, \ldots, q\}$, then
\[
\varphi_\pi(Z_1, \ldots, Z_n) = \prod_{t=1}^q \varphi(Z_{k_{t, 1}}\cdots Z_{k_{t, m_t}}).
\]

\begin{defn}
\label{def:kappa}
Let $(\mathcal{A}, E, \varepsilon)$ be a $B$-$B$-non-commutative probability space.  The \emph{operator-valued bi-free cumulant function} 
\[
\kappa : \bigcup_{n\geq 1} \bigcup_{\chi : \{1,\ldots, n\} \to \{\ell, r\}} BNC(\chi) \times \mathcal{A}_{\chi(1)} \times \cdots \times \mathcal{A}_{\chi(n)} \to B
\]
is defined by
\[
\kappa_\pi(Z_1,\ldots, Z_n) = \sum_{\substack{\sigma \in BNC(\chi) \\ \sigma \leq\pi}}  \E_\sigma(Z_1, \ldots, Z_n) \mu_{BNC}(\sigma, \pi)
\]
for each $\chi : \{1, \ldots, n\} \to \{\ell, r\}$, $\pi \in BNC(\chi)$, and $Z_k \in \mathcal{A}_{\chi(k)}$.
\end{defn}

Both the operator-valued moment and cumulant functions are special functions which \cite{CNS2014-2} calls bi-multiplicative.  Bi-multiplicative functions have reduction properties which allows one to compute their values once one knows their values on full bi-non-crossing partitions.  We refer the reader to \cite{CNS2014-2}*{Definition 4.2.1} for the rigorous definition of a bi-multiplicative function but one may heuristically think of a bi-multiplicative function based on the notion of a multiplicative function in free probability as follows.  Given $\pi \in BNC(\chi)$ and a bi-multiplicative map $\Phi$, each reduction property one may apply to $\Phi_\pi(Z_1, \ldots, Z_n)$ follows by
\begin{enumerate}
\item viewing the non-crossing partition $s_\chi^{-1}  \cdot \pi$,
\item rearranging the $n$-tuple $(Z_1, \ldots, Z_n)$ to $(Z_{s_\chi(1)}, \ldots, Z_{s_\chi(n)})$, 
\item replacing any occurrences of $L_bZ_j$, $Z_j L_b$, $R_b Z_j$, and $Z_j R_b$ with $bZ_j$, $Z_j b$, $Z_j b$, and $bZ_j$ respectively,
\item applying one of the properties of a multiplicative map from \cite{NSS2002}*{Section 2.2},
\item and reversing the above identifications.
\end{enumerate}  

Using the notion of bi-multiplicativity in the case that $B = \bC$, if $Z_1, \ldots, Z_n \in A$ and $\pi \in BNC(\chi)$ has blocks $V_t = \{k_{t, 1}<\cdots<k_{t,m_t}\}$ for $t \in \{1, \ldots, q\}$, then
\[
\kappa_\pi(Z_1, \ldots, Z_n) = \prod_{t=1}^q \kappa_{\pi|_{V_t}}((Z_1,\ldots, Z_n)|_{V_t}).
\]

\subsection{Bi-Free Families of Pairs of $\boldsymbol{B}$-Faces}

We are now in a position to discuss bi-freeness with amalgamation.
\begin{defn}
\label{defn:pair-of-B-faces}
Let $(\mathcal{A}, E_\mathcal{A}, \varepsilon)$ be a $B$-$B$-non-commutative probability space.  A \emph{pair of $B$-faces of} $\A$ is a pair $(C, D)$ of unital subalgebras of $\A$ such that
\[
\varepsilon(B \otimes 1_B) \subseteq C \subseteq \A_\ell \qquad \mathrm{and}\qquad \varepsilon(1_B \otimes B^{\mathrm{op}}) \subseteq D \subseteq \A_r.
\]

A family $\{(C_k, D_k)\}_{k \in K}$ of pair of $B$-faces of $\A$ is said to be \emph{bi-free with amalgamation over $B$} (or \emph{simply bi-free over $B$}) if there exist $B$-$B$-bimodules with specified $B$-vector states $\{(\X_k, \mathring{\X}_k, p_k)\}_{k \in K}$ and unital homomorphisms $l_k : C_k \to \L_{\ell}(\X_k)$ and $r_k : D_k \to \L_{r}(\X_k)$ such that the joint distribution of $\{(C_k, D_k)\}_{k \in K}$ with respect to $E_\mathcal{A}$ is equal to the joint distribution of the images of 
\[
\{((\lambda_k \circ l_k)(C_k), (\rho_k \circ r_k)(D_k))\}_{k \in K}
\]
inside $\L(\ast_{k \in K} \X_k)$ with respect to $E_{\L(\ast_{k \in K} \X_k)}$ where $\lambda_k$ and $\rho_k$ denote the left and right regular representation onto $\X_k \subseteq \ast_{k \in K} \X_k$ respectively.
\end{defn}

The following was the main result of \cite{CNS2014-2}.  In that which follows, note a map $\epsilon : \{1,\ldots, n\} \to K$ defines an element of $\P (n)$ whose blocks are $\{\epsilon^{-1}(\{k\})\}_{k \in K}$.
\begin{thm}[\cite{CNS2014-2}*{Theorem 7.1.4 and Theorem 8.1.1}]
\label{thm:bifree-classifying-theorem}
Let $(\mathcal{A}, E_\mathcal{A}, \varepsilon)$ be a $B$-$B$-non-commutative probability space and let $\{(C_k, D_k)\}_{k \in K}$ be a family of pairs of $B$-faces of $\A$.  Then $\{(C_k, D_k)\}_{k \in K}$ are bi-free with amalgamation over $B$ if and only if for all $\chi : \{1,\ldots, n\} \to \{\ell, r\}$, $\epsilon : \{1,\ldots, n\} \to K$, and 
\[
Z_k \in \left\{
\begin{array}{ll}
C_{\epsilon(k)} & \mbox{if } \chi(k) = \ell  \\
D_{\epsilon(k)} & \mbox{if } \chi(k) = r
\end{array} \right.,
\]
the formula
\begin{align}
E_{\A}(Z_1 \cdots Z_n) = \sum_{\pi \in BNC(\chi)} \left[ \sum_{\substack{\sigma\in BNC(\chi)\\\pi\leq\sigma\leq\epsilon}}\mu_{BNC}(\pi, \sigma) \right] \mathcal{E}_{\pi}(Z_1,\ldots, Z_n)   \label{eq:universal-polys}
\end{align}
holds.  Equivalently $\{(C_k, D_k)\}_{k \in K}$ are bi-free with amalgamation over $B$ if and only if
\[
\kappa_{1_\chi}(Z_1, \ldots, Z_n) = 0
\]
provided $\epsilon$ is not constant.
\end{thm}

\section{Classical and Free Independence in Bi-Free Probability}
\label{sec:ClassicalAndFreeIndependence}

This section further demonstrates how free and classical independence arise using the $(\ell, r)$-cumulants.   Our attention is restricted to the scalar setting in this section unless otherwise specified.

\subsection{Free Independence via $\boldsymbol{(\ell, r)}$-Cumulants}

It is not difficult to use the $(\ell, r)$-cumulants to construct the state for which algebras are freely independent.  Indeed let $\{A_k\}_{k \in K}$ be unital subalgebras of a non-commutative probability space $(\A,\varphi)$ and let $\psi$ be the unique state on $\ast_{k \in K} A_k$ determined by $\varphi$ for which $\{A_k\}_{k \in K}$ are freely independent (that is, $\psi = \ast_{k \in K} \varphi|_{A_k}$).  Then $\psi$ can be realized by summing up certain $(\ell, r)$-cumulants constructed from $\varphi$.  Indeed for all $\epsilon : \{1,\ldots, n\} \to K$ and for all $Z_k \in A_{\epsilon(k)}$, 
\begin{align}
\psi(Z_1 \cdots Z_n) = \sum_{\substack{\pi \in BNC(\chi) \\ \pi \leq \epsilon}} \kappa_\pi(Z_1, \ldots, Z_n) \label{eq:free-via-lr-cumulants}
\end{align}
where $\chi : \{1, \ldots, n\} \to \{\ell, r\}$ is either constant map and $\pi\leq \epsilon$ denotes $\pi$ is a refinement of the partition with blocks $\{\epsilon^{-1}(\{k\})\}_{k \in K}$ (that is, $\pi$ may be coloured via $\epsilon$).  Indeed the above formula holds as the $(\ell, r)$-cumulants reduce to the free cumulants when $\chi$ is constant.  Note that the above generalizes to the operator-valued setting when we restrict each $Z_k$ to be an element of $\A_\ell$.

\subsection{Classical Independence via $\boldsymbol{(\ell, r)}$-Cumulants}

It is also possible to use the $(\ell, r)$-cumulants to construct the state for which a pair of algebras are classically independent.  To do so, we need the following collection of bi-non-crossing partitions.
\begin{defn}
Given a map $\chi : \{1,\ldots, n\} \to \{\ell, r\}$, a bi-non-crossing partition $\pi \in BNC(\chi)$ is said to be \emph{vertically split} if whenever $V$ is a block of $\pi$, either $V \subseteq \chi^{-1}(\{\ell\})$ or $V \subseteq \chi^{-1}(\{r\})$. The set of vertically split bi-non-crossing partitions is denoted by $\bncs(\chi)$.
\end{defn}
Notice if $(\A, \varphi)$ is a non-commutative probability space, $Z_1, \ldots, Z_n \in A$, and $\chi : \{1, \ldots, n\} \to \{\ell, r\}$ is such that
\[
\chi^{-1}(\{\ell\}) = \{i_1 < i_2 < \cdots < i_k\} \qqand \chi^{-1}(\{r\}) = \{j_1 < j_2 < \cdots < j_m\}
\]
then
\[
\sum_{\pi \in \bncs(\chi)} \kappa_\pi(Z_1,\ldots, Z_n) = \varphi(Z_{i_1} Z_{i_2} \cdots Z_{i_k}) \varphi(Z_{j_1} Z_{j_2} \cdots Z_{j_m})
\]
since
\[
\kappa_\pi(Z_1, \ldots, Z_n) = \kappa_{\pi|_V}((Z_1, \ldots, Z_n)|_V) \kappa_{\pi|_{V^c}}((Z_1,\ldots, Z_n)|_{V^c})
\]
whenever $V$ is a union of blocks of $\pi$.  In particular, if $A_1$ and $A_2$ are unital subalgebras of $\A$ and $\psi$ is the unique state on $A_1 \ast A_2$ determined by $\varphi$ for which $A_1$ and $A_2$ are classically independent (that is, $\psi = \varphi|_{A_1} \otimes \, \varphi|_{A_2}$), then for all $\epsilon : \{1,\ldots, n\} \to \{1,2\}$ and for all $Z_k \in A_{\epsilon(k)}$
\begin{align}
\psi(Z_1 \cdots Z_n) = \sum_{\pi \in \bncs(\chi_\epsilon)} \kappa_\pi(Z_1, \ldots, Z_n) = \sum_{\substack{\pi \in BNC(\chi_\epsilon) \\ \pi \leq \epsilon}} \kappa_\pi(Z_1, \ldots, Z_n) \label{eq:independent-via-LR-cumulants}
\end{align}
where $\chi_\epsilon : \{1, \ldots, n\} \to \{\ell, r\}$ is defined by
\[
\chi_\epsilon(k) = \left\{
\begin{array}{ll}
\ell & \mbox{if } \epsilon(k) = 1  \\
r & \mbox{if } \epsilon(k) = 2 
\end{array} \right. .
\]

In the operator-valued setting, given a $B$-$B$-non-commutative probability space $(\A, E, \varepsilon)$, unital subalgebras $A_1 \subseteq \A_\ell$ and $A_2 \subseteq \A_r$, $\epsilon : \{1,\ldots, n\} \to \{1,2\}$, and $Z_k \in A_{\epsilon(k)}$, it can be shown using the properties of bi-multiplicative functions that
\[
\sum_{\pi \in \bncs(\chi_\epsilon)} \kappa_\pi(Z_1, \ldots, Z_n) = E(Z_{i_1} \cdots Z_{i_k}) E(Z_{j_1} \cdots Z_{j_m})
\]
where
\[
\{i_1 < i_2 < \cdots < i_k\} = \chi^{-1}_\epsilon(\{\ell\}) \qqand  \{j_1 < j_2 < \cdots < j_m\} = \chi^{-1}_\epsilon(\{r\}).
\]

\subsection{The Bi-Free Kac/Loeve Theorem} 
\label{sec:Kac-Loeve-Theorem}

In \cite{N1996}*{Theorem 5.3} the free Kac/Loeve Theorem was proved demonstrating that any pair of freely independent random variables for which a non-trivial rotation remained free must have been free central limit distributions; that is, semicircular variables.  Said theorem follows by the linearity of each entry of the free cumulants.

In \cite{V2013-1}*{Theorem 7.4} it was shown that the bi-free central limits distributions arise precisely when all $(\ell, r)$-cumulants of order at least three vanish.  In particular, if a pair $(T, S)$ is a bi-free central limit distribution, then there are four values to specify:
\[
\kappa_{(\ell, \ell)}(T, T), \quad \kappa_{(\ell, r)}(T, S), \quad \kappa_{(r, \ell)}(S, T), \qand \kappa_{(r,r)}(S, S).
\]

The following generalizes \cite{N1996}*{Theorem 5.3} to the bi-free setting.  One can use the same arguments to generalize \cite{N1996}*{Theorem 5.1} to the bi-free setting as well.

\begin{thm}[Bi-Free Kac/Loeve Theorem]
\label{thm:Kac-Loeve}
Let $(\A, \varphi)$ be a non-commutative probability space.  Suppose $(T_1, S_1)$ and $(T_2, S_2)$ are bi-free two-faced families in $\A$ such that $\varphi(T_k) = 0 = \varphi(S_k)$ for $k\in \{1,2\}$.  For a fixed $\theta \in (0, \frac{\pi}{2})$, let
\begin{align*}
T_3 = \cos(\theta) T_1 + \sin(\theta) T_2, &\quad &
S_3 &= \cos(\theta) S_1 + \sin(\theta) S_2,\\
T_4 = -\sin(\theta) T_1 + \cos(\theta) T_2, &\qquad\text{and} &
S_4 &=-\sin(\theta) S_1 + \cos(\theta) S_2.
\end{align*}
If $(T_3, S_3)$ and $(T_4, S_4)$ are bi-freely independent, then $(T_1, S_1)$ and $(T_2, S_2)$ must be bi-free central limit distributions with equal second order $(\ell, r)$-cumulants.

Conversely, if $(T_1, S_1)$ and $(T_2, S_2)$ are bi-free two-faced families in $\A$ and are bi-free central limit distributions with equal second order $(\ell, r)$-cumulants, then $(T_3, S_3)$ and $(T_4, S_4)$ are bi-freely independent.
\end{thm}
\begin{proof}
The proof easily follows from Theorem \ref{thm:bifree-classifying-theorem} (or simply \cite{CNS2014-1}*{Theorem 4.3.1}) and the linearity of the bi-free cumulants in each entry.  For $m \in \{1,2,3,4\}$ and $k \in \{\ell, r\}$, let
\[
Z_{k,m}  = \left\{
\begin{array}{ll}
T_m & \mbox{if } k = \ell  \\
S_m & \mbox{if } k=r
\end{array} \right. .
\]

Suppose $(T_3, S_3)$ and $(T_4, S_4)$ are bi-freely independent.  To see that  $(T_1, S_1)$ and $(T_2, S_2)$ have equal second order $(\ell, r)$-cumulants, let $\chi : \{1,2\} \to \{\ell, r\}$ be arbitrary.  Then
\begin{align*}
0 &= \kappa_{1_\chi}(Z_{\chi(1), 3}, Z_{\chi(2), 4}) = -\cos(\theta)\sin(\theta) \kappa_{1_\chi}(Z_{\chi(1), 1}, Z_{\chi(2), 1}) +\cos(\theta)\sin(\theta) \kappa_{1_\chi}(Z_{\chi(1), 2}, Z_{\chi(2), 2}).
\end{align*}
Since $\theta \in (0, \frac{\pi}{2})$, we obtain 
\[
\kappa_{1_\chi}(Z_{\chi(1), 1}, Z_{\chi(2), 1}) = \kappa_{1_\chi}(Z_{\chi(1), 2}, Z_{\chi(2), 2}).
\]

To see all higher-order $(\ell, r)$-cumulants of $(T_1, S_1)$ and $(T_2, S_2)$ are zero, let $\chi : \{1, \ldots, n\} \to \{\ell, r\}$ for $n \geq 3$ be arbitrary.  Then
\begin{align*}
0 &= \kappa_{1_\chi}(Z_{\chi(1), 3}, Z_{\chi(2), 3}, \ldots, Z_{\chi(n-2), 3}, Z_{\chi(n-1), 3}, Z_{\chi(n), 4}) \\
&= -\cos^{n-1}(\theta) \sin(\theta)\kappa_{1_\chi}(Z_{\chi(1), 1}, \ldots,  Z_{\chi(n), 1}) +\sin^{n-1}(\theta) \cos(\theta)\kappa_{1_\chi}(Z_{\chi(1), 2}, \ldots,  Z_{\chi(n), 2})
\end{align*}
and
\begin{align*}
0 &=  \kappa_{1_\chi}(Z_{\chi(1), 3}, Z_{\chi(2), 3}, \ldots, Z_{\chi(n-2), 3}, Z_{\chi(n-1), 4}, Z_{\chi(n), 4}) \\
 &= \cos^{n-2}(\theta) \sin^2(\theta)\kappa_{1_\chi}(Z_{\chi(1), 1}, \ldots,  Z_{\chi(n), 1})  +\sin^{n-2}(\theta) \cos^2(\theta)\kappa_{1_\chi}(Z_{\chi(1), 2}, \ldots,  Z_{\chi(n), 2}).
\end{align*}
Since the matrix
\[
\left[  \begin{array}{cc} -\cos^{n-1}(\theta) \sin(\theta) & \sin^{n-1}(\theta) \cos(\theta) \\  \cos^{n-2}(\theta) \sin^2(\theta) & \sin^{n-2}(\theta) \cos^2(\theta)  \end{array} \right]
\]
has determinant $-\sin^n(\theta) \cos^n(\theta)$, which is non-zero as $\theta \in (0, \frac{\pi}{2})$, the above system of equations imply
\[
\kappa_{1_\chi}(Z_{\chi(1), 1}, \ldots,  Z_{\chi(n), 1}) = 0 = \kappa_{1_\chi}(Z_{\chi(1), 2}, \ldots,  Z_{\chi(n), 2}).
\]

For the converse, one can easily use the fact that $(T_1, S_1)$ and $(T_2, S_2)$ are bi-freely independent and bi-free central limit distributions to show that all mixed $(\ell, r)$-cumulants of $(T_3, S_3)$ and $(T_4, S_4)$ of order at least three vanish.  In addition, since $(T_1, S_1)$ and $(T_2, S_2)$ are bi-freely independent and have equal second order cumulants, for all $\chi : \{1,2\} \to \{\ell, r\}$
\begin{align*}
\kappa_{1_\chi}(Z_{\chi(1), 3}, Z_{\chi(2), 4}) = -\cos(\theta)\sin(\theta) \kappa_{1_\chi}(Z_{\chi(1), 1}, Z_{\chi(2), 1}) +\cos(\theta)\sin(\theta) \kappa_{1_\chi}(Z_{\chi(1), 2}, Z_{\chi(2), 2}) =0
\end{align*}
and similarly $\kappa_{1_\chi}(Z_{\chi(1), 4}, Z_{\chi(2), 3}) = 0$. Hence Theorem \ref{thm:bifree-classifying-theorem} implies $(T_3, S_3)$ and $(T_4, S_4)$ are bi-freely independent.
\end{proof}

It is natural to ask whether Theorem \ref{thm:Kac-Loeve} holds when the left operators undergo one rotation and the right operators undergo a different rotation.  The above computations demonstrate the following which, in general, is the best one can hope for.
\begin{prop}
Let $(\A, \varphi)$ be a non-commutative probability space.  Suppose $(T_1, S_1)$ and $(T_2, S_2)$ are bi-free two-faced families in $\A$ such that $\varphi(T_k) = 0 = \varphi(S_k)$ for $k\in \{1,2\}$.  For fixed $\theta_\ell, \theta_r \in (0, \frac{\pi}{2})$, let
\begin{align*}
T_3 = \cos(\theta_\ell) T_1 + \sin(\theta_\ell) T_2, &\quad &
S_3 &= \cos(\theta_r) S_1 + \sin(\theta_r) S_2,\\
T_4 = -\sin(\theta_\ell) T_1 + \cos(\theta_\ell) T_2, &\qquad\text{and} &
S_4 &=-\sin(\theta_r) S_1 + \cos(\theta_r) S_2.
\end{align*}
If $(T_3, S_3)$ and $(T_4, S_4)$ are bi-freely independent, then $(T_1, S_1)$ and $(T_2, S_2)$ must be bi-free central limit distributions.
\end{prop}

\section{Boolean Independence in Bi-Free Probability}
\label{sec:Boolean}

This section demonstrates how operator-valued Boolean independence arises and can be studied in the bi-free setting.  It is advised for the reader to keep the scalar case $B = \bC$ in mind as things simplify slightly.

\subsection{Boolean Independent Algebras from Bi-Free Pairs of Faces}

We begin by recalling the definitions for operator-valued Boolean independent algebras.
\begin{defn}
Let $\A$ be a unital algebra containing $B$ (with $1_\A= 1_B$).  A (possibly non-unital) subalgebra $A \subseteq \A$ is said to be a \emph{$B$-algebra} if $BAB \subseteq A$.
\end{defn}

\begin{defn}
\label{defn:Boolean}
Let $(\A, \Phi)$ be a $B$-non-commutative probability space and let $A_1, \ldots, A_n$ be $B$-algebras contained in $\A$.  We say that $A_1, \ldots, A_n$ are \emph{Boolean independent with amalgamation over $B$} (or simply \emph{Boolean independent over $B$}) if 
\[
\Phi(Z_1 \cdots Z_n) = \Phi(Z_1) \cdots \Phi(Z_n)
\]
whenever $Z_m \in A_{k_m}$ are such that $k_m \neq k_{m+1}$ for all $m \in \{1,\ldots, n-1\}$.
\end{defn}

The following demonstrates a method for producing equations like those in Definition \ref{defn:Boolean} by taking bi-free pairs of $B$-faces and operators that are alternating products of left and right operators from the same $B$-face.  We make the choice of `left before right' as one needs to use $B^{\mathrm{op}}$ for the `right before left' option (note this second option works in the case $B = \bC$).
\begin{lem}
\label{lem:bi-free-that-looks-like-Boolean}
Let $\{(C_k, D_k)\}_{k \in K}$ be bi-free pairs of $B$-faces in a $B$-$B$-non-commutative probability space $(\A, E_\A, \varepsilon)$.  Let $k_1, \ldots, k_n \in K$ be such that $k_m \neq k_{m+1}$, let $\{q_m\}^n_{m=1}\subseteq \bN$, and let 
\[
\{T_{m, t}\}^{q_m}_{t=1} \subseteq C_{k_m} \qqand \{S_{m, t}\}^{q_m}_{t=1} \subseteq D_{k_m} \cap \A_\ell
\]
be such that
\[
E_\A(   T_{m, 1} T_{m,2} \cdots T_{m,q_m}) = 0 = E_\A(S_{m, 1} S_{m,2} \cdots S_{m,q_m}).
\]
If
\[
Z_m = T_{m, 1} S_{m, 1} \cdots T_{m, q_m} S_{m, q_m},
\]
then
\[
E_\A(Z_1 \cdots Z_n) = E_\A(Z_1) \cdots E_\A(Z_n).
\]
\end{lem}
\begin{proof}
Since $\{(C_k, D_k)\}_{k \in K}$ are bi-free pairs of $B$-faces, there exists $B$-$B$-bimodules with specified $B$-vector states $(\X_k, \mathring{\X}_k, p_k)$ and unital homomorphisms
\[
\alpha_k : C_k \to \L_\ell(\X_k) \qqand \beta_k : D_k \to \L_r(\X_k)
\]
such that if
\[
\lambda_k : \L_\ell(\X_k) \to  \L_\ell(\ast_{m \in K} (\X_m, \mathring{\X}_m, p_m)) \AND \rho_k : \L_r(\X_k) \to  \L_r(\ast_{m \in K} (\X_m, \mathring{\X}_m, p_m))
\]
are the left and right regular representations respectively, $E$ is the expectation of $\L(\ast_{m \in K} (\X_m, \mathring{\X}_m, p_m))$ onto $B$, 
\begin{align*}
T'_{m,p} = \lambda_{k_m}(&\alpha_{k_m}(T_{m,p})), \quad S'_{m,p} = \rho_{k_m}(\beta_{k_m}(S_{m,p})), \quad \text{and} \quad Z'_m = T'_{m, 1} S'_{m, 1} \cdots T'_{m, q_m} S'_{m, q_m},
\end{align*}
then
\begin{align*}
&E(   T'_{m, 1} T'_{m,2} \cdots T'_{m,q_m}) = 0 \text{ for all }m,\\
&E(S'_{m, 1} S'_{m,2} \cdots S'_{m,q_m}) =0\text{ for all }m,\\
&E(Z'_m) = E_\A(Z_m) \text{ for all $m$, and}\\
&E(Z'_1 \cdots Z'_n) = E_\A(Z_1 \cdots Z_n).
\end{align*}

Let
\begin{align*}
\eta_0 &= 1_B \in \ast_{m \in K} (\X_m, \mathring{\X}_m, p_m), \\
\eta_m &= Z'_m \eta_0 - E(Z'_m)\eta_0 \in \mathring{\X}_{k_m}, \\
\zeta_m &= T'_{m, 1} T'_{m,2} \cdots T'_{m,q_m} (1_B) \in \mathring{\X}_{k_m}, \text{ and}\\
\omega_m &= S'_{m, 1} S'_{m,2} \cdots S'_{m,q_m} (1_B) \in \mathring{\X}_{k_m}.
\end{align*}
We claim that if
\[
\prod^n_{m=a+1} E(Z'_m) := E(Z'_{a+1}) E(Z'_{a+2}) \cdots E(Z'_{n})
\]
then
\begin{align*}
Z'_1 & \cdots Z'_n (\eta_0) = \sum^n_{a=0}    \zeta_1 \otimes \zeta_2 \otimes \cdots \otimes \zeta_{a-1} \otimes R_{\left[\prod^n_{m=a+1} E(Z'_m) \right] }(\eta_a) \otimes \omega_{a-1} \otimes \cdots \otimes \omega_1
\end{align*}
from which the lemma clearly follows.  It is important to note
\[
R_{\left[\prod^n_{m=a+1} E(Z'_m) \right] }(\eta_a) \in \mathring{X}_{k_a}.
\]

To see the claim, we proceed by induction on $n$.  The case $n=1$ is trivial as
\[
Z'_1(\eta_0) = E(Z'_1) \oplus \eta_1 = R_{E(Z'_1)} (1_B) \oplus \eta_1.
\]

To proceed inductively, suppose the result holds for $n-1$.  In particular, by relabelling, we may assume that
\begin{align*}
Z'_2 \cdots Z'_n  (\eta_0) = R_{\left[\prod^n_{m=2} E(Z'_m) \right]} (\eta_0) 
+ \sum^n_{a=2}     \zeta_2 \otimes \zeta_3 \otimes \cdots \otimes \zeta_{a-1} \otimes R_{\left[\prod^n_{m=a+1} E(Z'_m) \right]}(\eta_a) \otimes \omega_{a-1}  \otimes \cdots \otimes \omega_2.
\end{align*}
by the induction hypothesis.  For $a \geq 2$, notice
\begin{align*}
&Z'_1 \left(\zeta_2 \otimes \zeta_3 \otimes \cdots \otimes \zeta_{a-1} \otimes R_{\left[\prod^n_{m=a+1} E(Z'_m) \right]}(\eta_a) \otimes \omega_{a-1} \otimes \cdots \otimes \omega_2\right)\\ 
&=  \zeta_1 \otimes \zeta_2 \otimes \zeta_3 \otimes \cdots \otimes \zeta_{a-1} \otimes R_{\left[\prod^n_{m=a+1} E(Z'_m) \right]}(\eta_a) \otimes \omega_{a-1} \otimes \cdots \otimes \omega_2 \otimes \omega_1.
\end{align*}
In addition, since $\{S_{m, t}\}^{q_m}_{t=1} \subseteq D_{k_m} \cap \A_\ell$, one obtains $Z'_1 \in \A_\ell$ so
\begin{align*}
Z'_1  R_{\left[\prod^n_{m=2} E(Z'_m) \right]} (\eta_0) &=  R_{\left[\prod^n_{m=2} E(Z'_m) \right]} Z'_1(\eta_0)\\
&=  R_{\left[\prod^n_{m=2} E(Z'_m) \right]}  \left(E(Z'_1) \oplus \eta_1 \right) \\
&= \left(R_{\left[\prod^n_{m=2} E(Z'_m) \right]} R_{E(Z'_1)} (1_B)   \right) \oplus R_{\left[\prod^n_{m=2} E(Z'_m) \right]} (\eta_1)\\
&= \left(R_{\left[\prod^n_{m=1} E(Z'_m) \right]} (\eta_0)   \right) \oplus R_{\left[\prod^n_{m=2} E(Z'_m) \right]} (\eta_1).
\end{align*}
Hence the claim and lemma follow.
\end{proof}

Lemma \ref{lem:bi-free-that-looks-like-Boolean} easily enables the construction of $B$-algebras which are Boolean independent over $B$ from bi-free pairs of $B$-faces.  To begin the construction, recall from Remark \ref{rem:B-ncps-from-B-B-ncps} that if $(\A, E, \varepsilon)$ is a $B$-$B$-non-commutative probability space, then $(\A_\ell, E)$ is a $B$-non-commutative probability space where $\varepsilon(B \otimes 1_B)$ is the copy of $B$.  If $(C, D)$ is a pair of $B$-faces, $C' \subseteq C$, $D' \subseteq D \cap \A_\ell$, and
\[
C' D' := \{TS \, \mid \, T \in C', S \in D'\}
\]
then $\alg(C'D') \subseteq \A_\ell$ (where $\alg(X)$ represents the (not necessarily unital) algebra generated by $X$).  If $L_b C' \subseteq C'$ and $C' L_b \subseteq C'$ for all $b \in B$, then it is clear that $\alg(C'D')$ is a $B$-algebra contained in $\A_\ell$.  Using Lemma \ref{lem:bi-free-that-looks-like-Boolean}, we immediately obtain the following.

\begin{thm}
\label{thm:algebras-are-Boolean-independent}
Let $\{(C_k, D_k)\}_{k \in K}$ be bi-free pairs of $B$-faces in a $B$-$B$-non-commutative probability space $(\A, E, \varepsilon)$.  For each $k \in K$, let 
\[
C'_k \subseteq C_k \qqand D'_k \subseteq D_k \cap \A_\ell
\]
be subsets such that $L_b C'_k, C'_k L_b \subseteq C'_k$ for all $b \in B$ and
\[
E((C'_k)^n) = \{0\} = E((D'_k)^n)
\]
for all $n \geq 1$.  Then $\{\alg(C'_kD'_k)\}_{k \in K}$ are $B$-algebras in the $B$-non-commutative probability space $(\A_\ell, E)$ that are Boolean independent over $B$.
\end{thm}

One may be concern that the sets $D_k \cap \A_\ell$ might just be scalars.  Clearly this is not the case when $B = \bC$ and we show an instance in the operator-valued setting where the intersection is non-empty in Construction \ref{cons:Boolean}.  In particular, consider the following example.

\begin{exam}
Let $\mathbb{F}_n$ denote the free group on $n$ generators $u_1, \ldots, u_n$ and let $\varphi$ be the vector state on $\B(\ell_2(\mathbb{F}_n))$ corresponding to the point mass at the identity.  If $\lambda, \rho : \mathbb{F}_n \to \B(\ell_2(\mathbb{F}_n))$ denote the left and right regular representations respectively, recall $\{(\lambda(u_k), \rho(u_k))\}_{k =1}^n$ are bi-free two-faced families with respect to $\varphi$.  Hence Theorem \ref{thm:algebras-are-Boolean-independent} implies that 
\[
\left\{ \alg \left( \left\{ \lambda(u_k^p) \rho(u_k^q) \, \mid \, p,q \in \mathbb{N} \right\} \right)  \right\}_{k=1}^n
\]
are Boolean independent with respect to $\varphi$ and
\[
\left\{ \alg \left( \left\{ \lambda(u_k^p) \rho(u_k^{-q}) \, \mid \, p,q \in \mathbb{N} \right\} \right)  \right\}_{k=1}^n
\]
are Boolean independent with respect to $\varphi$.
\end{exam}

\begin{rem}
\label{rem:simple-counter-example}
The converse of Theorem \ref{thm:algebras-are-Boolean-independent} does not hold even in the scalar setting: if $\{(C'_k, D'_k)\}_{k \in K}$ are pairs of faces in a non-commutative probability space $(\A, \varphi)$ with 
\[
\varphi((C'_k)^n) = \{0\} = \varphi((D'_k)^n)
\]
for all $n\geq 1$ and $k \in K$, then the Boolean independence of $\{\alg(C'_kD'_k)\}_{k \in K}$ and the Boolean independence of $\{\alg(D'_kC'_k)\}_{k \in K}$ is not enough to guarantee that 
\[
\{(\bC 1_\A + \alg(C'_k), \bC 1_\A +\alg(D'_k))\}_{k \in K}
\]
are bi-free.

For a concrete example where this converse fails, consider $(M_2(\bC), \tau)$ where $\tau$ is the normalized trace on $M_2(\bC)$.  Let
\[
T = \left[  \begin{array}{cc} 0 & 1 \\ 0 & 0 \end{array} \right], \qquad S = \left[  \begin{array}{cc} 0 & 0 \\ 1 & 0 \end{array} \right],
\]
$C'_1 = D'_1 = \mathbb{C} T$, and $C'_2 = D'_2 = \mathbb{C}S$.  It is clear that $\{(C'_k, D'_k)\}_{k =1, 2}$ are pairs of algebras that have the specified properties.  However, for 
\[
\{(\bC 1_\A + \alg(C'_k), \bC 1_\A +\alg(D'_k))\}_{k =1,2}
\]
to be bi-free, one would require for any $\chi : \{1,2\} \to \{\ell, r\}$ that
\[
0 = \kappa_{1_\chi}(T, S) = \tau(TS) - \tau(T)\tau(S).
\]
It is clear that $\tau(TS) = \frac{1}{2}$ whereas $\tau(T) = \tau(S) = 0$ demonstrating the above line does not hold.
\end{rem}

\begin{rem}
\label{rem:operator-model-counterexample}
The main issue with Remark \ref{rem:simple-counter-example} is that $C'_1$ and $D'_2$ are not classically independent with respect to $\varphi$.  However, even the additional conditions that $\{\bC 1_\A + C'_k\}_{k\in K}$ are freely independent, that $\{\bC 1_\A + D'_k\}_{k \in K}$ are freely independent, and that $\bC 1_\A + C'_{k_1}$ and $\bC 1_\A + D'_{k_2}$ are classically independent for all $k_1, k_2 \in K$ is not enough to guarantee that 
\[
\{(\bC 1_\A + \alg(C'_k), \bC 1_\A +\alg(D'_k))\}_{k \in K}
\]
are bi-free.

To see the above claim, note by the operator model in \cite{CNS2014-1} there exists a non-commutative probability space $(\A, \varphi)$ and operators $T_1, T_2, S_1, S_2 \in \A$ such that $T_1, T_2$ are left operators, $S_1, S_2$ are right operators, and all $(\ell, r)$-cumulants involving these operators are zero except $\kappa_{\chi}(T_1, T_2, S_2) =1$.  If $C'_k = \alg(\{T_k\})$ and $D'_k = \alg(\{S_k\})$, then 
\[
\{(\bC 1_\A + \alg(C'_k), \bC 1_\A +\alg(D'_k))\}_{k \in \{1,2\}}
\]
is not a bi-freely independent family with respect to $\varphi$ by Theorem \ref{thm:bifree-classifying-theorem}.  However, clearly equations (\ref{eq:free-via-lr-cumulants}) and (\ref{eq:independent-via-LR-cumulants}) of Section \ref{sec:ClassicalAndFreeIndependence} imply
\[
\varphi((C'_k)^n) = \{0\} = \varphi((D'_k)^n)
\]
for all $n\geq 1$ and $k \in K$,  $\bC 1_\A + C'_1$ and $\bC 1_\A + C'_2$ are freely independent, $\bC 1_\A + D'_1$ and $\bC 1_\A + D'_2$ are freely independent, and $\bC 1_\A + C'_{k_1}$ and $\bC 1_\A + D'_{k_2}$ are classically independent for all $k_1, k_2 \in \{1,2\}$.  To see that $\{\alg(C'_kD'_k)\}_{k \in \{1,2\}}$ are Boolean independent, one must show that
\begin{align*}
\varphi  \left( \left(\prod^{m_1}_{m=1} T^{p_{1,m}}_{k_1} S^{q_{1,m}}_{k_1}\right) \cdots \left(\prod^{m_n}_{m=1} T^{p_{n,m}}_{k_n} S^{q_{n,m}}_{k_n}\right)  \right)  
= \varphi\left(\prod^{m_1}_{m=1} T^{p_{1,m}}_{k_1} S^{q_{1,m}}_{k_1}\right) \cdots \varphi\left(\prod^{m_n}_{m=1} T^{p_{n,m}}_{k_n} S^{q_{n,m}}_{k_n}\right)
\end{align*}
for all $n\geq 2$, $m_1, \ldots m_n \in \mathbb{N}$, $p_{k,m}, q_{k,m} \in \mathbb{N}$, and $k_1,\ldots, k_n \in \{1,2\}$ with $k_m \neq k_{m+1}$.  Notice 
\[
\varphi\left(\prod^{t}_{m=1} T^{p_{m}}_{k} S^{q_{m}}_{k}\right) = 0
\]
since all $(\ell, r)$-cumulants involving $T_k$ and $S_k$ are zero. In addition
\[
\varphi \left( \left(\prod^{m_1}_{m=1} T^{p_{1,m}}_{k_1} S^{q_{1,m}}_{k_1}\right) \cdots \left(\prod^{m_n}_{m=1} T^{p_{n,m}}_{k_n} S^{q_{n,m}}_{k_n}\right)  \right) = 0
\]
since the above expression is a sum of products of $(\ell, r)$-cumulants where each product of $(\ell, r)$-cumulants must contain at least one involving $S_1$ and thus is zero.  Similarly $\{\alg(D'_kC'_k)\}_{k \in \{1,2\}}$ are Boolean independent thus completing the claim.
\end{rem}

\subsection{Bi-Free Boolean Systems}

To examine Boolean independence over $B$ inside bi-free probability, we restrict ourselves to the following abstract structure. 
\begin{defn}
\label{defn:bi-free-Boolean-system}
Let $\{(C_k, D_k)\}_{k \in K}$ be bi-free pairs of $B$-faces in a $B$-$B$-non-commutative probability space $(\A, E, \varepsilon)$.  For each $k \in K$, let 
\[
C'_k \subseteq C_k \qqand D'_k \subseteq D_k\cap \A_\ell
\]
be subsets such that $L_b C'_k, C'_k L_b \subseteq C'_k$ for all $b \in B$.  We say that $\{(C'_k, D'_k)\}_{k \in K}$ is a \emph{bi-free Boolean $B$-system with respect to $E$} if
\begin{enumerate}
\item \label{cond:nilpotent-Boolean} $(C'_k)^2 = \{0\} = (D'_k)^2$ for all $k \in K$,
\item \label{cond:left-followedby-alternating-zero-moments}$E(C'_k (D'_kC'_k)^n) = \{0\}$ for all $n \geq 0$ and $k \in K$, and
\item \label{cond:right-followedby-alternating-zero-moments}$E(D'_k (C'_kD'_k)^n) = \{0\}$ for all $n \geq 0$ and $k \in K$.
\end{enumerate}
\end{defn}

Note Theorem \ref{thm:algebras-are-Boolean-independent} directly implies the following.
\begin{cor}
\label{cor:boolean-systems-give-boolean-independent-algebras}
Let $\{(C'_k, D'_k)\}_{k \in K}$ be a bi-free Boolean $B$-system with respect to $E$.  Then $\{\alg(C'_kD'_k)\}_{k \in K}$ are Boolean independent over $B$ with respect to $E$.
\end{cor}

The reason for Definition \ref{defn:bi-free-Boolean-system} is the following construction showing any collection of $B$-algebras which are Boolean independent over $B$ may be realized inside a bi-free Boolean $B$-system.
\begin{cons}
\label{cons:Boolean}
Let $(\mathcal{A}, \Phi)$ be a $B$-non-commutative probability space.  By Remark \ref{rem:B-ncps-from-B-B-ncps} we may view $\A$ as a $B$-$B$-module with specified $B$-vector state which we denote by $(\A, \mathring{\A}, \Phi)$ where $\mathring{\A} = \ker(\Phi)$.  Consider $\A \oplus \A$ as a $B$-$B$-bimodule via the action
\[
b_1 \cdot (Z_1 \oplus Z_2) \cdot b_2 = (b_1 Z_1 b_2) \oplus (b_1 Z_2 b_2)
\]
for all $b_1, b_2 \in B$ and $Z_1, Z_2 \in \A$.  Notice $\A \oplus \A$ then becomes a $B$-$B$-module with specified $B$-vector state via the triple $(\A \oplus \A, \mathring{\A} \oplus \A, \Psi)$ where
\[
\Psi(Z_1 \oplus Z_2) = \Phi(Z_1).
\]

We need to consider some special operators in $\L(\A \oplus \A)$.  For $Z \in \A$ define $T_Z \in \L(\A \oplus \A)$ by
\[
T_Z(Z_1 \oplus Z_2) = ZZ_2 \oplus 0.
\]
Then $T_Z \in \L_\ell(\A \oplus \A)$ since
\begin{align*}
T_Z R_b (Z_1 \oplus Z_2) = T_Z (Z_1b \oplus Z_2b) = ZZ_2b \oplus 0 = R_b (ZZ_2 \oplus 0) = R_b T_Z (Z_1 \oplus Z_2)
\end{align*}
for all $b \in B$ and $Z_1, Z_2 \in \A$.  In addition, notice
\[
T_Z L_b (Z_1 \oplus Z_2) = T_Z( bZ_1 \oplus bZ_2) = ZbZ_2 \oplus 0 = T_{Zb}(Z_1 \oplus Z_2)
\]
and
\[
L_b T_Z(Z_1 \oplus Z_2) = L_b(ZZ_2 \oplus 0) = b Z Z_2 \oplus 0 = T_{bZ} (Z_1 \oplus Z_2)
\]
for all $b \in B$ and $Z_1, Z_2 \in \A$.  Hence $L_b T_Z = T_{bZ}$ and $T_Z L_b = T_{Zb}$.  Furthermore, clearly 
\[
T_{z_1 Z_1 + z_2 Z_2} = z_1 T_{Z_1} + z_2 T_{z_2}
\]
for all $z_1, z_2 \in \bC$ and $Z_1, Z_2 \in \A$.

In addition, define $S_{1_B} \in \L(\A \oplus \A)$ by
\[
S_{1_B}(Z_1 \oplus Z_2) = 0 \oplus Z_1.
\]
Then $S_{1_B} \in \L_\ell(\A \oplus \A) \cap \L_r(\A \oplus \A)$ since
\begin{align*}
S_{1_B} L_{b_1} R_{b_2} (Z_1 \oplus Z_2) &= S_{1_B} (b_1Z_1b_2 \oplus b_1Z_2b_2)\\
 &= 0 \oplus b_1 Z_1 b_2 \\
&=L_{b_1} R_{b_2} (0 \oplus Z_1) \\
&= L_{b_1} R_{b_2} S_{1_B} (Z_1 \oplus Z_2)
\end{align*}
for all $b_1, b_2 \in B$ and $Z_1, Z_2 \in \A$.

Let $\{A_k\}_{k \in K}$ be $B$-algebras contained in $\A$ that are Boolean independent over $B$ with respect to $\Phi$.   For each $k \in K$, we can consider a copy of $(\A \oplus \A, \mathring{\A} \oplus \A, \Psi)$, denoted $(\Y_k, \mathring{\Y}_k, \Psi_k)$.  Let $E$ denote the expectation of $\L(\ast_{m \in K}(\Y_m, \mathring{\Y}_m, \Psi_k))$ onto $B$ and let
\begin{align*}
\lambda_k &: \L_\ell(\Y_k) \to \L_\ell(\ast_{m \in K} (\Y_m, \mathring{\Y}_m, \Psi_m)) \quad \text{and} \\
\rho_k &: \L_r(\Y_k) \to \L_r(\ast_{m \in K} (\Y_m, \mathring{\Y}_m, \Psi_m))
\end{align*}
be the left and right regular representations onto the $k^{\text{th}}$ term respectively.  By definition the pairs of $B$-faces $\{(\lambda_k(\L_\ell(\Y_k)), \rho_k(\L_r(\Y_k)))\}$ are bi-free inside the $B$-$B$-non-commutative probability space $\L(\ast_{m \in K} (\Y_m, \mathring{\Y}_m, \Psi_m))$.

For each $k \in K$ and $Z \in A_k$, consider the elements
\[
T_{k, Z} = \lambda_k(T_Z) \qqand S_{k, 1_B} = \rho_k(S_{1_B}).
\]
It is elementary to verify that if for each $k \in K$
\[
C'_k = \left\{ T_{k,Z} \, \mid \, Z \in A_k \right\} \qqand D'_k = \left\{ S_{k,1_B} \right\},
\]
then $\{(C'_k, D'_k)\}_{k \in K}$ is a bi-free Boolean $B$-system in $\{(\lambda_k(\L(\Y_k)), \rho_k(\L(\Y_k)))\}$.  In addition, it is clear by construction that for a fixed $k \in K$, if $Z_1, \ldots, Z_n \in A_k$ then
\begin{align}
E(T_{k,Z_1} S_{k,1_B} T_{k,Z_2} S_{k, 1_B} \cdots T_{k,Z_n} S_{k,1_B}) = \Phi(Z_1 \cdots Z_n). \label{eq:moment-expression-in-Boolean-construction}
\end{align}
\end{cons}

Combining Corollary \ref{cor:boolean-systems-give-boolean-independent-algebras} along with Construction \ref{cons:Boolean} we obtain the following.
\begin{thm}
\label{thm:Boolean-embeds-into-bi-free-Boolean-system}
Given a collection $\{A_k\}_{k \in K}$ of Boolean independent $B$-algebras in a $B$-non-commutative probability space $(\A_0, \Phi)$, there exists a bi-free Boolean $B$-system $\{(C'_k, D'_k)\}_{k \in K}$ inside a $B$-$B$-non-commutative probability space $(\A, E, \varepsilon)$ and injective $B$-linear maps $\beta_k : A_k \to C'_kD'_k$ such that 
\[
E(\beta_{k_1}(Z_1) \cdots \beta_{k_n}(Z_n)) = \Phi(Z_1 \cdots Z_n)
\]
for all $Z_m \in A_{k_m}$ and for all $k_m \in K$.  
\end{thm}
\begin{proof}
Using the notation of Construction \ref{cons:Boolean}, define $\beta_k : A_k \to C'_kD'_k$ by 
\[
\beta_k(Z) = T_{k,Z} S_{k,1_B}.
\]
Thus $\beta_k$ is linear and
\[
\beta_k(b_1Zb_2) = T_{k,b_1 Z b_2} S_{k,1_B} = L_{b_1} T_{k,Z} L_{b_2} S_{k,1_B} = L_{b_1} T_{k,Z} S_{k,1_B} L_{b_2}
\]
for all $b_1, b_2 \in B$ and $Z \in A_k$.  The result then follows from Corollary \ref{cor:boolean-systems-give-boolean-independent-algebras} and equation (\ref{eq:moment-expression-in-Boolean-construction}).
\end{proof}
Although Theorem \ref{thm:Boolean-embeds-into-bi-free-Boolean-system} may feel slightly unsatisfactory since the $\beta_k$ are, in general, not homomorphisms, the result enables the study of Boolean independence with amalgamation through larger Boolean independent algebras.  In particular, we show in Section \ref{sec:Boolean-Cumulants-Via-LR-Cumulants} that the operator-valued Boolean cumulant functions may be recovered through the operator-valued bi-free cumulant function and Theorem \ref{thm:Boolean-embeds-into-bi-free-Boolean-system}.  On the other hand, we note the following method for embedding Boolean independent algebras into algebras produced by bi-free pairs of $B$-faces using homomorphisms.
\begin{cons}
Let $(\mathcal{A}, \Phi)$ be a $B$-non-commutative probability space.  Viewing $(\A \oplus \A, \mathring{\A} \oplus \A, \Psi)$ as a $B$-$B$-module with specified $B$-vector state as in Construction \ref{cons:Boolean}, for $Z \in \A$ consider the special operators $T'_Z, U_{1_B} \in \L(\A \oplus \A)$ defined by
\[
T'_Z(Z_1 \oplus Z_2) = 0 \oplus ZZ_2 \qqand U_{1_B}(Z_1 \oplus Z_2) = Z_2 \oplus Z_1.
\]
Then $T'_Z \in \L_\ell(\A \oplus \A)$ and $U_{1_B} \in \L_\ell(\A \oplus \A) \cap \L_r(\A \oplus \A)$.  Furthermore, $T'_{z_1Z_1+z_2Z_2} = z_1T'_{Z_1} + z_2 T'_{Z_2}$ and $T'_{b_1 Z b_2} = L_{b_1} T'_Z L_{b_2}$ for all $z_1, z_2 \in \bC$, $b_1, b_2 \in B$, and $Z_1, Z_2, Z \in \A$.

Let $\{A_k\}_{k \in K}$ be $B$-algebras contained in $\A$ that are Boolean independent over $B$ with respect to $\Phi$.   Let $(\Y_k, \mathring{\Y}_k, \Psi_k)$, $E$, $\lambda_k$, and $\rho_k$ be as in Construction \ref{cons:Boolean}.

For each $k \in K$, consider the map $\beta_k : A_k \to \rho_k(\L_r(\Y_k)) \lambda_k(\L_\ell(\Y_k)) \rho_k(\L_r(\Y_k))$ defined by
\[
\beta_k(Z) = \rho_k(U_{1_B}) \lambda_k(T'_Z) \rho_k(U_{1_B}).
\]
Since $\rho_k$ and $\lambda_k$ are unital homomorphisms, and since $U_{1_B}^2 = I_{\A \oplus \A}$, $\beta_k$ is a homomorphism.  Finally, if $Z_m \in A_{k_m}$ with $k_m \neq k_{m+1}$ and $\eta_m =  (Z_m - \Phi(Z_m) 1_B) \oplus 0 \in \mathring{\Y}_{k_m}$, then, using an inductive argument similar to that in Lemma \ref{lem:bi-free-that-looks-like-Boolean}, we obtain for all $m \in \{1,\ldots, n\}$ that
\begin{align*}
\beta_{k_m}(Z_m) \cdots \beta_{k_n}(Z_n)(1_B \oplus 0) = \Phi(Z_1) \cdots \Phi(Z_n) (1_B \oplus 0) +  R_{\Phi(Z_{m+1}) \cdots \Phi(Z_{n})}(\eta_m).
\end{align*}
Therefore
\[
E(\beta_{k_1}(Z_1) \cdots \beta_{k_n}(Z_n)) = \Phi(Z_1) \cdots \Phi(Z_n) = \Phi(Z_1 \cdots Z_n).
\]
Hence the $B$-algebras $\{\beta_k(A_k)\}_{k \in K}$ are Boolean independent over $B$ with respect to $E$.
\end{cons}

\subsection{Boolean Bi-Non-Crossing Partitions}

To further our study of bi-free Boolean $B$-systems, we need to examine specific collections of bi-non-crossing partitions.  
\begin{defn}
A map $\chi : \{1,\ldots, 2n\}\to \{\ell,r\}$ is said to be \emph{alternating} if 
\[
\chi(k) = \left\{
\begin{array}{ll}
\ell & \mbox{if } k \mbox{ is odd}  \\
r & \mbox{if } k \mbox{ is even}
\end{array} \right. .
\]
Given an alternating map $\chi$, a bi-non-crossing partition $\pi \in BNC(\chi)$ is said to be \emph{Boolean} if $2k-1$ and $2k$ are in the same block of $\pi$ for all $k \in \{1,\ldots, n\}$.  The set of Boolean bi-non-crossing partitions is denoted by $\bncb(\chi)$.
\end{defn}

\begin{exam}
For the alternating $\chi : \{1,\ldots, 6\} \to \{\ell, r\}$, the elements of $\bncb(\chi)$ may be represented via the following bi-non-crossing diagrams.
\begin{align*}
	\begin{tikzpicture}[baseline]
			\draw[thick,dashed] (-.5,1.5) -- (-.5,-.25) -- (.5,-.25) -- (.5,1.5);
			\node[left] at (-.5,1.25) {1};
			\draw[black,fill=black] (-.5,1.25) circle (0.05);
			\node[right] at (.5,1) {2};
			\draw[black,fill=black] (.5,1) circle (0.05);
			\node[left] at (-.5,.75) {3};
			\draw[black,fill=black] (-.5,.75) circle (0.05);
			\node[right] at (.5,.5) {4};
			\draw[black,fill=black] (.5,.5) circle (0.05);
			\node[left] at (-.5,.25) {5};
			\draw[black,fill=black] (-.5,.25) circle (0.05);
			\node[right] at (.5,0) {6};
			\draw[black,fill=black] (.5,0) circle (0.05);
			\draw[thick, black] (.5,0) -- (0,0) -- (0, .25) -- (-.5,.25);
			\draw[thick, black] (.5,.5) -- (0,.5) -- (0, .75) -- (-.5,.75);
			\draw[thick, black] (.5,1) -- (0,1) -- (0, 1.25) -- (-.5,1.25);
	\end{tikzpicture}
	\qquad
	\begin{tikzpicture}[baseline]
			\draw[thick,dashed] (-.5,1.5) -- (-.5,-.25) -- (.5,-.25) -- (.5,1.5);
			\node[left] at (-.5,1.25) {1};
			\draw[black,fill=black] (-.5,1.25) circle (0.05);
			\node[right] at (.5,1) {2};
			\draw[black,fill=black] (.5,1) circle (0.05);
			\node[left] at (-.5,.75) {3};
			\draw[black,fill=black] (-.5,.75) circle (0.05);
			\node[right] at (.5,.5) {4};
			\draw[black,fill=black] (.5,.5) circle (0.05);
			\node[left] at (-.5,.25) {5};
			\draw[black,fill=black] (-.5,.25) circle (0.05);
			\node[right] at (.5,0) {6};
			\draw[black,fill=black] (.5,0) circle (0.05);
			\draw[thick, black] (.5,0) -- (0,0) -- (0, .25) -- (-.5,.25);
			\draw[thick, black] (.5,.5) -- (0,.5) -- (0, .75) -- (-.5,.75);
			\draw[thick, black] (.5,1) -- (0,1) -- (0, 1.25) -- (-.5,1.25);
			\draw[thick, black] (0,.25) -- (0, .5);
	\end{tikzpicture}
	\qquad
	\begin{tikzpicture}[baseline]
			\draw[thick,dashed] (-.5,1.5) -- (-.5,-.25) -- (.5,-.25) -- (.5,1.5);
			\node[left] at (-.5,1.25) {1};
			\draw[black,fill=black] (-.5,1.25) circle (0.05);
			\node[right] at (.5,1) {2};
			\draw[black,fill=black] (.5,1) circle (0.05);
			\node[left] at (-.5,.75) {3};
			\draw[black,fill=black] (-.5,.75) circle (0.05);
			\node[right] at (.5,.5) {4};
			\draw[black,fill=black] (.5,.5) circle (0.05);
			\node[left] at (-.5,.25) {5};
			\draw[black,fill=black] (-.5,.25) circle (0.05);
			\node[right] at (.5,0) {6};
			\draw[black,fill=black] (.5,0) circle (0.05);
			\draw[thick, black] (.5,0) -- (0,0) -- (0, .25) -- (-.5,.25);
			\draw[thick, black] (.5,.5) -- (0,.5) -- (0, .75) -- (-.5,.75);
			\draw[thick, black] (.5,1) -- (0,1) -- (0, 1.25) -- (-.5,1.25);
			\draw[thick, black] (0,1) -- (0, .75);
	\end{tikzpicture}
	\qquad
	\begin{tikzpicture}[baseline]
			\draw[thick,dashed] (-.5,1.5) -- (-.5,-.25) -- (.5,-.25) -- (.5,1.5);
			\node[left] at (-.5,1.25) {1};
			\draw[black,fill=black] (-.5,1.25) circle (0.05);
			\node[right] at (.5,1) {2};
			\draw[black,fill=black] (.5,1) circle (0.05);
			\node[left] at (-.5,.75) {3};
			\draw[black,fill=black] (-.5,.75) circle (0.05);
			\node[right] at (.5,.5) {4};
			\draw[black,fill=black] (.5,.5) circle (0.05);
			\node[left] at (-.5,.25) {5};
			\draw[black,fill=black] (-.5,.25) circle (0.05);
			\node[right] at (.5,0) {6};
			\draw[black,fill=black] (.5,0) circle (0.05);
			\draw[thick, black] (.5,0) -- (0,0) -- (0, .25) -- (-.5,.25);
			\draw[thick, black] (.5,.5) -- (0,.5) -- (0, .75) -- (-.5,.75);
			\draw[thick, black] (.5,1) -- (0,1) -- (0, 1.25) -- (-.5,1.25);
			\draw[thick, black] (0,.25) -- (0, .5);
			\draw[thick, black] (0,1) -- (0, .75);
	\end{tikzpicture}
\end{align*}
\end{exam}

\begin{rem}
Note if $\chi : \{1,\ldots, 2n\} \to \{\ell, r\}$ is alternating, then $\bncb(\chi)$ is naturally isomorphic to $\I(n)$; the lattice of interval partitions on $\{1,\ldots, n\}$.  Indeed define $\Psi : \I(n) \to \bncb(\chi)$ as follow: if 
\[
\pi = \{\{t_{k}, t_{k}+1, \ldots, t_{k+1} - 1\}\}^{m-1}_{k=1}
\]
for some sequence $t_1 = 1 < t_2 < t_3 < \cdots < t_m = n+1$, define
\[
\Psi(\pi) = \{\{2t_{k}-1, 2t_{k}, \ldots, 2t_{k+1} - 2\}\}^{m-1}_{k=1}.
\]
We do not use this relation in that which follows and derive the Boolean cumulant functions independently via $\bncb(\chi)$.
\end{rem}

\begin{rem}
\label{rem:partial-mobius-inversion-in-bnc-Boolean}
Fix $\chi : \{1,\ldots, 2n\} \to \{\ell, r\}$ alternating.  It is then clear that $\bncb(\chi)$ is a sublattice of $BNC(\chi)$ with maximal element $1_\chi$ and minimal element $0_{b,\chi}$ whose blocks are $\{\{2k-1, 2k\}\}_{k \in K}$.  As such, one can restrict $\mu_{BNC}$ to $\bncb$.  In particular, given $\pi, \sigma \in \bncb(\chi)$,
\[
\sum_{\substack{\tau \in \bncb(\chi) \\\pi \leq \tau \leq \sigma}} \mu_{BNC}(\tau, \sigma) = \sum_{\substack{\tau \in \bncb(\chi) \\ \pi \leq \tau \leq \sigma}} \mu_{BNC}(\pi, \tau) = \left\{
\begin{array}{ll}
1 & \mbox{if } \pi = \sigma  \\
0 & \mbox{otherwise}
\end{array} \right. .
\]
Hence there is a M\"{o}bius inversion inside $\bncb(\chi)$ by the proof of \cite{NS2006}*{Proposition 10.11}: if $X_\ell$ and $X_r$ are sets and
\[
f, g : \bigcup_{n\geq 1} \bigcup_{\substack{\chi : \{1,\ldots, 2n\} \to \{\ell, r\} \\ \chi \text{ alternating}}} \bncb(\chi) \times X_{\chi(1)} \times \cdots \times X_{\chi(n)} \to B
\]
are such that
\[
f(\pi, x_1, \ldots, x_{2n}) = \sum_{\substack{\sigma \in \bncb(\chi) \\ \sigma \leq \pi}} g(\sigma , x_1, \ldots, x_{2n})
\]
for all $\pi \in \bncb(\chi)$ and $x_k \in X_{\chi(k)}$, then
\[
g(1_\chi, x_1, \ldots, x_{2n}) = \sum_{\pi \in \bncb(\chi)} f(\pi, x_1, \ldots, x_{2n}) \mu_{BNC}(\pi, 1_\chi)
\]
for all $x_k \in X_{\chi(k)}$ and alternating $\chi : \{1,\ldots, 2n\} \to \{\ell, r\}$.
\end{rem}

\subsection{Boolean Cumulants via Bi-Free Operator-Valued Cumulants}
\label{sec:Boolean-Cumulants-Via-LR-Cumulants}

This section demonstrates only certain bi-free operator-valued cumulants (those corresponding to Boolean bi-non-crossing partitions) are necessary in order to compute the joint moments of elements from bi-free Boolean $B$-systems.  In particular, we develop the analogue of equations (\ref{eq:free-via-lr-cumulants}) and (\ref{eq:independent-via-LR-cumulants}) from Section \ref{sec:ClassicalAndFreeIndependence} in the Boolean setting and provide an alternate definition for the Boolean cumulant functions.  All results in this section make use of the following technical lemma.
\begin{lem}
\label{lem:certain-moment-terms-vanish-in-Boolean-systems}
Let $\{(C'_k, D'_k)\}_{k \in K}$ be a bi-free Boolean $B$-system in a $B$-$B$-non-commutative probability space $(\A, E, \varepsilon)$.  Let $\chi : \{1,\ldots, 2n\} \to \{\ell, r\}$ be alternating and let $\epsilon : \{1,\ldots, 2n\} \to K$ be such that $\epsilon(2m-1) = \epsilon(2m)$ for all $m \in \{1,\ldots, n\}$.  Furthermore, for each $k \in K$, let $T_k \in C'_{\epsilon(2k-1)}$ and let $S_k \in D'_{\epsilon(2k)}$.  If $\pi \in BNC(\chi)$ and $\pi \leq \epsilon$, then
\[
\mathcal{E}_\pi(T_1, S_1, \ldots, T_n, S_n) = 0
\]
unless $\pi \in \bncb(\chi)$. 
\end{lem}
\begin{proof}
To simplify notation, let 
\[
\Theta_\pi = \mathcal{E}_\pi(T_1, S_1, \ldots, T_n, S_n).
\]
Notice by bi-multiplicativity that $\Theta_\pi = 0$ if $\pi$ has any blocks of cardinality one by conditions (\ref{cond:left-followedby-alternating-zero-moments}) and (\ref{cond:right-followedby-alternating-zero-moments}) in Definition \ref{defn:bi-free-Boolean-system}.  We claim that if $\Theta_\pi \neq 0$, then 1 and 2 must be in the same block of $\pi$.  To see this, suppose otherwise that 1 and 2 are in different blocks of $\pi$.  We divide the proof into two cases:

\underline{\textit{Case (1): The block $V$ of $\pi$ containing 1 contains a $k$ with $\chi(k) = r$.}}   Let 
\[
m_0 = \min\{m \, \mid \, 2m \in V, m \in \{1,\ldots, n\}\}.
\]
Note $m_0 \neq \infty$ and $m_0 \geq 2$ by the assumptions in this case.  Hence, since $\pi \in BNC(\chi)$,
\[
W = \{2m  \, \mid \,  m < m_0\}\subseteq \chi^{-1}(\{r\})
\]
is a non-empty union of blocks of $\pi$ disjoint from $V$ that is a $\chi$-interval.  Thus $\pi|_{W}$ is a non-crossing partition on $W$.  Writing $W = \{t_1 < t_2 < \cdots < t_q\}$ one sees, by using conditions (\ref{cond:nilpotent-Boolean}) and (\ref{cond:right-followedby-alternating-zero-moments}) along with the fact that right $B$-faces in bi-free pairs of $B$-faces are freely independent over $B$, that 
\[
\E_{\pi|_W}((T_1, S_1, T_2, S_2, \cdots, T_n, S_n)|_W) = 0.
\]
Thus $\Theta_\pi = 0$ by bi-multiplicativity.

\underline{\textit{Case (2): The block $V$ of $\pi$ containing 1 contains a $k$ with $k \neq 1$ and $\chi(k) = \ell$.}}  
Let 
\[
m_0 = \min\{m \, \mid \, 2m-1 \in V, m \in \{2,\ldots, n\}\}.
\]
Note $m_0 \neq \infty$ by the assumptions in this case.  If $m_0 = 2$, then $\epsilon(1) = \epsilon(3)$ since $\pi \leq \epsilon$.  Hence, since $2 \notin V$, we obtain by bi-multiplicativity that $\Theta_\pi =0$ as $T_1T_3 = 0$ by condition (\ref{cond:nilpotent-Boolean}).  Thus we may assume that $m_0 \geq 2$.  Since $\pi \in BNC(\chi)$, 
\[
W = \{2m - 1 \, \mid \, 1 < m < m_0\}\subseteq \chi^{-1}(\{\ell\})
\]
is a non-empty union of blocks of $\pi$ disjoint from $V$ that is a $\chi$-interval.  Thus $\pi|_{W}$ is a non-crossing partition on $W$.  Writing $W = \{t_1 < t_2 < \cdots < t_q\}$ one sees,  by using conditions (\ref{cond:nilpotent-Boolean}) and (\ref{cond:left-followedby-alternating-zero-moments}) of Definition \ref{defn:bi-free-Boolean-system} along with the fact that left $B$-faces in bi-free pairs of $B$-faces are freely independent over $B$, that 
\[
\E_{\pi|_W}((T_1, S_1, T_2, S_2, \cdots, T_n, S_n)|_W) = 0.
\]
Thus $\Theta_\pi = 0$ by bi-multiplicativity.

Cases (1) and (2) show that $1$ and $2$ must be in the same block of $\pi$ in order for $\Theta_\pi$ to be non-zero.  Suppose we have shown that $2k-1$ and $2k$  must be in the same block of $\pi$ for all $k \in \{1,\ldots, k_0 - 1\}$ in order for $\Theta_\pi$ to be non-zero.  We claim under this supposition that $\Theta_\pi \neq 0$ implies $2k_0 - 1$ and $2k_0$ must be in the case block of $\pi$.  To see this, suppose $2k_0 - 1$ and $2k_0$ are in different blocks of $\pi$.

If there exists an $m \in \{1,\ldots, 2k_0 - 2\}$ such that $2k_0 $ and $m$ are in the same block $W$ of $\pi$, then there exists a $k \in \{1,\ldots, k_0 - 1\}$ such that 
\[
2k-1, 2k, 2k+1, \ldots, 2k_0-2, 2k_0 \in W
\]
and $W \cap \{1,\ldots, 2k-2\} = \emptyset$.   In this case $\epsilon(q) = \epsilon(2k_0)$ for all $q \in \{2k-1, 2k, \ldots, 2k_0-2\}$ since $\pi \leq \epsilon$.  Therefore $\Theta_\pi = 0$ by bi-multiplicativity and since $T_{2k_0-2}T_{2k_0} = 0$ by condition (\ref{cond:nilpotent-Boolean}) of Definition \ref{defn:bi-free-Boolean-system}.  

Next, if there exists an $m \in \{1,\ldots, 2k_0 - 2\}$ such that $2k_0 - 1$ and $m$ are in the same block $W$ of $\pi$, then, by assumptions and since $\pi \in BNC(\chi)$, there exists a $k \in \{1,\ldots, k_0 - 1\}$ such that 
\[
2k-1, 2k, 2k+1, \ldots, 2k_0-1 \in W
\]
and $W \cap \{1,\ldots, 2k-2\} = \emptyset$.  In this case $\epsilon(q) = \epsilon(2k_0-1)$ for all $q \in \{2k-1, 2k, \ldots, 2k_0-2\}$ since $\pi \leq \epsilon$.  If $W = \{2k-1, 2k, \ldots, 2k_0 - 1\}$, then, by condition (\ref{cond:left-followedby-alternating-zero-moments}) of Definition \ref{defn:bi-free-Boolean-system}, by the fact that $T_k, S_k \in \A_\ell$, and by bi-multiplicativity, one obtains $\Theta_\pi = 0$.  If $W \neq \{2k-1, 2k, \ldots, 2k_0 - 1\}$, one can use the previous paragraph and arguments similar to those in cases (1) and (2) to show that $\Theta_\pi = 0$.  

Otherwise $\{1, 2, \ldots, 2k_0-2\}$ has empty intersection with the blocks of $\pi$ containing $2k_0-1$ and $2k_0$.  Therefore, by arguments identical to those in cases (1) and (2), $\Theta_\pi = 0$.  Hence the claim and thus proof is complete.
\end{proof}

Note the following produces the analogue to equations (\ref{eq:free-via-lr-cumulants}) and (\ref{eq:independent-via-LR-cumulants}) of Section \ref{sec:ClassicalAndFreeIndependence} for Boolean independence.
\begin{thm}
\label{thm:Boolean-formulae-work-as-they-should}
Let $\{(C'_k, D'_k)\}_{k \in K}$ be a bi-free Boolean $B$-system with respect to $E$.  If $\chi : \{1,\ldots, 2n\} \to \{\ell, r\}$ is alternating, $\epsilon : \{1,\ldots, 2n\} \to K$ is such that $\epsilon(2m-1) = \epsilon(2m)$ for all $m \in \{1,\ldots, n\}$, $T_k \in C'_{\epsilon(2k-1)}$, and $S_k \in D'_{\epsilon(2k)}$, then
\begin{align}
E(T_1S_1 \cdots T_nS_n) = \sum_{\pi \in \bncb(\chi)} \kappa_\pi(T_1, S_1, \ldots, T_n, S_n). \label{eq:Boolean-independent-via-LR-cumualnts}
\end{align}
Equivalently 
\[
\kappa_{1_\chi}(T_1, S_1, \ldots, T_n, S_n) = \sum_{\pi \in \bncb(\chi)} \E_\pi(T_1, S_1, \ldots, T_n, S_n) \mu_{BNC}(\pi, 1_\chi).
\]
Furthermore $\kappa_{1_\chi}(T_1, S_1, T_2, S_2, \ldots, T_n, S_n) = 0$ unless $\epsilon$ is constant.
\end{thm}
\begin{proof}
For the first claim, we notice by Definition \ref{def:kappa}, equation (\ref{eq:universal-polys}) in Theorem \ref{thm:bifree-classifying-theorem}, Remark \ref{rem:partial-mobius-inversion-in-bnc-Boolean}, and Lemma \ref{lem:certain-moment-terms-vanish-in-Boolean-systems} that
\begin{align*}
E(T_1S_1 \cdots T_nS_n) &=  \sum_{\pi \in BNC(\chi)} \left[ \sum_{\substack{ \sigma \in BNC(\chi) \\\pi \leq \sigma \leq \epsilon  }} \mu_{BNC}(\pi, \sigma)   \right] \E_\pi(T_1, S_1, \ldots, T_n, S_n) \\
&= \sum_{\pi \in \bncb(\chi)} \left[ \sum_{\substack{ \sigma \in BNC(\chi) \\\pi \leq \sigma \leq \epsilon  }}  \mu_{BNC}(\pi, \sigma)  \right] \E_\pi(T_1, S_1, \ldots, T_n, S_n) \\
&= \sum_{\pi \in \bncb(\chi)} \left[ \sum_{\substack{ \sigma \in \bncb(\chi) \\\pi \leq \sigma \leq \epsilon  }}  \mu_{BNC}(\pi, \sigma)  \right] \E_\pi(T_1, S_1, \ldots, T_n, S_n) \\
&= \sum_{\substack{\sigma \in \bncb(\chi) \\ \sigma \leq \epsilon}} \sum_{\substack{ \pi \in \bncb(\chi) \\\pi \leq \sigma }}\E_\pi(T_1, S_1, \ldots, T_n, S_n) \mu_{BNC}(\pi, \sigma) \\
&= \sum_{\substack{\sigma \in \bncb(\chi) \\ \sigma \leq \epsilon}} \sum_{\substack{ \pi \in BNC(\chi) \\\pi \leq \sigma }}\E_\pi(T_1, S_1, \ldots, T_n, S_n) \mu_{BNC}(\pi, \sigma) \\
&= \sum_{\sigma \in \bncb(\chi)} \kappa_\sigma(T_1, S_1, T_2, S_2, \ldots, T_n, S_n).
\end{align*}
The second claim follows from the first by the M\"{o}bius inversion in Remark \ref{rem:partial-mobius-inversion-in-bnc-Boolean} and the third equation follows since mixed operator-valued bi-free cumulants vanish by Theorem \ref{thm:bifree-classifying-theorem}.  
\end{proof}

\begin{rem}
\label{rem:boolean-cumulant-via-LR}
Using Theorem \ref{thm:Boolean-formulae-work-as-they-should}, we can easily construct a new definition for the operator-valued Boolean cumulant functions in terms of the operator-valued bi-free cumulant function.  Let $\{A_k\}_{k \in K}$ be $B$-algebras in a $B$-non-commutative probability space $(\A, \Phi)$ that are Boolean independent over $B$.  By Theorem \ref{thm:Boolean-embeds-into-bi-free-Boolean-system} and by Theorem \ref{thm:Boolean-formulae-work-as-they-should}, given $Z_m \in A_{k_m}$ we define the Boolean cumulant functions via the formula
\[
\kappa_{\text{B}}(Z_1, \ldots, Z_n) = \kappa_{1_\chi}(T_{k_1, Z_1}, S_{k_1, 1_B}, \ldots, T_{k_n, Z_n}, S_{k_n, 1_B}).
\]
In particular, the Boolean cumulants can be realized as $(\ell, r)$-cumulants.
\end{rem}

\section{Monotone Independence in Bi-Free Probability}
\label{sec:Monotone}

This section demonstrates how operator-valued monotone independence arises and can be studied inside bi-free probability.

\subsection{Monotone Independent Subalgebras from Bi-Free Pairs of Faces}

We begin by recalling the definition of monotone independence over $B$.
\begin{defn}
Let $(\A, \Phi)$ be a $B$-non-commutative probability space and let $(A_\lambda)_{\lambda \in \Lambda}$ be $B$-algebras contained in $\A$.  Given a linear ordering $<$ on $\Lambda$, the collection $(A_\lambda)_{\lambda \in \Lambda}$ is said to be \emph{monotonically independent with amalgamation over $B$ with respect to $(\Phi, <)$} (or simply \emph{monotonically independent over $B$}) if
\[
\Phi(Z_1 Z_2 \cdots Z_{k-1}Z_kZ_{k+1} \cdots Z_n) = \Phi(Z_1 Z_2 \cdots Z_{k-1} \Phi(Z_k) Z_{k+1} \cdots Z_n)
\]
whenever $Z_m \in A_{\lambda_m}$, $\lambda_{k} > \lambda_{k-1}$, and $\lambda_k > \lambda_{k+1}$ (where one inequality is irrelevant when $k = 1$ or $k = n$).  

Similarly, the collection $(A_\lambda)_{\lambda \in \Lambda}$ is said to be \emph{anti-monotonically independent over $B$ with respect to $(\Phi, <)$} if $(A_\lambda)_{\lambda \in \Lambda}$ are monotonically independent over $B$ with respect to $(\Phi, >)$.
\end{defn}
We only deal with the case $\Lambda = \{1,2\}$ equipped the natural ordering $1 < 2$ in which case we simply say that $A_1$ and $A_2$ are monotonically independent over $B$.

The following demonstrates a way to construct monotonically independent $B$-algebras from bi-free pairs of $B$-faces in the spirit of Theorem \ref{thm:algebras-are-Boolean-independent}.
\begin{thm}
\label{thm:monotone-independence-of-algebras}
Let $\{(C_1, D_1), (C_2, D_2)\}$ be bi-free pairs of $B$-faces in a $B$-$B$-non-commutative probability space $(\A, E, \varepsilon)$.  Suppose $D'_1 \subseteq D_1 \cap \A_\ell$ and $C'_1 \subseteq C_1$ is such that 
\[
L_b C'_1, C'_1 L_b \subseteq C'_1  \text{ for all }b \in B\qqand E((C'_1)^n) = \{0\} \text{ for all }n.
\]
Then $\alg(C'_1D'_1)$ and $C_2$ are $B$-algebras that are monotonically independent over $B$ in the $B$-non-commutative probability space $(\A_\ell, E)$.
\end{thm}
\begin{proof}
Recall by Remark \ref{rem:B-ncps-from-B-B-ncps} that $(\A_\ell, E)$ is a $B$-non-commutative probability space with $\varepsilon(B \otimes 1_B)$ as the copy of $B$.  Recall
\[
\varepsilon(B \otimes 1_B) \subseteq C_2 \subseteq \A_\ell
\]
so $C_2$ is a $B$-algebra of $\A_\ell$.  In addition, by the conditions on $C'_1$ and $D'_1$, it is clear that $\alg(C'_1D'_1)$ is a $B$-algebra.

To show that $\alg(C'_1D'_1)$ and $C_2$ are monotonically independent over $B$, it suffices to show that
\begin{align*}
E &\left(Z_0 \left( \prod^{m_1}_{k=1} T_{1,k}S_{1,k} \right)Z_1 \left( \prod^{m_2}_{k=1} T_{2,k}S_{2,k} \right) \cdots \left( \prod^{m_n}_{k=1} T_{n,k}S_{n,k} \right) Z_n     \right) \\
&= E\left(L_{E(Z_0)} \left( \prod^{m_1}_{k=1} T_{1,k}S_{1,k} \right)L_{E(Z_1)} \left( \prod^{m_2}_{k=1} T_{2,k}S_{2,k} \right) \cdots \left( \prod^{m_n}_{k=1} T_{n,k}S_{n,k} \right) L_{E(Z_n)}    \right)
\end{align*}
for all $n \in \mathbb{N}$, $Z_0, Z_1, \ldots, Z_n \in C_2$, $m_1, \ldots, m_n \in \mathbb{N}$, $\{\{T_{q,k}\}^{m_q}_{k=1}\}^n_{q=1} \subseteq C'_1$, and $\{\{S_{q,k}\}^{m_q}_{k=1}\}^n_{q=1} \subseteq D'_1$.  Let 
\[
\chi : \left\{1, \ldots, 1 + n + 2\sum^{n}_{k=1} m_k\right\} \to \{\ell, r\}
\]
be determined by the above sequence of operators; that is, $\chi(m) = r$ if $m$ is the index of some $S_{q,k}$ in the above sequence of operators and $\chi(m) = \ell$ otherwise.  Similarly, let 
\[
\epsilon: \left\{1, \ldots, 1 + n + 2\sum^{n}_{k=1} m_k\right\} \to \{1,2\}
\]
be determined by the above sequence of operators; that is, $\epsilon(m) = 2$ if $m$ is the index of some $Z_k$ in the above sequence of operators and $\epsilon(m) = 1$ otherwise.  Therefore, since $\{(C_1, D_1), (C_2, D_2)\}$ is a bi-free family, 
\begin{align*}
E &\left(Z_0 \left( \prod^{m_1}_{k=1} T_{1,k} S_{1,k} \right)Z_1 \left( \prod^{m_2}_{k=1} T_{2,k} S_{2,k} \right) \cdots Z_{n-1}\left( \prod^{m_n}_{k=1} T_{n,k} S_{n,k} \right) Z_n     \right) \\
&= \sum_{\substack{\pi \in BNC(\chi) \\ \pi \leq \epsilon}} \kappa_\pi(Z_0, T_{1,1}, S_{1,1}, T_{1,2}, S_{1,2}, \ldots, T_{n,m_n}, S_{n, m_n}, Z_n).
\end{align*}

We claim that if $\pi \in BNC(\chi)$ is such that $\pi \leq \epsilon$ and $\pi$ contains a block containing at least two indices corresponding to different $Z_k$, then
\[
\kappa_\pi(Z_0, T_{1,1}, S_{1,1}, T_{1,2}, S_{1,2}, \ldots, T_{n,m_n}, S_{n, m_n}, Z_n) = 0.
\]
Indeed by Definition \ref{def:kappa}
\begin{align*}
\kappa_\pi&(Z_0, T_{1,1}, S_{1,1}, T_{1,2}, S_{1,2}, \ldots, T_{n,m_n}, S_{n, m_n}, Z_n) \\
&= \sum_{\substack{\sigma \in BNC(\chi) \\ \sigma \leq \pi}} \mathcal{E}_\sigma(Z_0, T_{1,1}, S_{1,1}, T_{1,2}, S_{1,2}, \ldots, T_{n,m_n}, S_{n, m_n}, Z_n) \mu_{BNC}(\sigma, \pi).
\end{align*}
However due to the nature of bi-non-crossing partitions, each $\sigma \leq \pi$ has a block $W$ that is a $\chi$-interval and all of whose indices are elements of $C'_1$.  Therefore, since $E((C'_1)^n) = \{0\}$ for all $n$, 
\[
\mathcal{E}_{\sigma|_W}((Z_0, T_{1,1}, S_{1,1}, T_{1,2}, S_{1,2}, \ldots, T_{n,m_n}, S_{n, m_n}, Z_n)|_W) = 0
\]
and bi-multiplicativity implies
\[
\mathcal{E}_\sigma(Z_0, T_{1,1}, S_{1,1}, T_{1,2}, S_{1,2}, \ldots, T_{n,m_n}, S_{n, m_n}, Z_n) = 0
\]
so the claim follows.  

Therefore every $\pi \in BNC(\chi)$ which has a non-zero contribution to sum in the moment expression has the indices corresponding to each $Z_k$ as singletons.  If
\[
\chi_0 : \left\{1, \ldots, 2\sum^n_{k=1} m_k\right\} \to \{\ell, r\}
\]
is alternating, then by restricting to non-zero contributions in the sum and using the properties of bi-multiplicative functions, we obtain
\begin{align*}
&E \left(Z_0 \left( \prod^{m_1}_{k=1} T_{1,k} S_{1,k} \right)Z_1 \left( \prod^{m_2}_{k=1} T_{2,k} S_{2,k} \right) \cdots Z_{n-1}\left( \prod^{m_n}_{k=1} T_{n,k} S_{n,k} \right) Z_n     \right) \\
&= \sum_{\pi_0 \in BNC(\chi_0)} \kappa_{\pi_0}(L_{E(Z_0)}T_{1,1}, S_{1,1}, \ldots, L_{E(Z_1)} T_{2,1}, S_{2,1}, \ldots, T_{n,m_n}, S_{n, m_n}R_{E(Z_n)}) \\
&= E\left(L_{E(Z_0)} \left( \prod^{m_1}_{k=1} T_{1,k}S_{1,k} \right)L_{E(Z_1)} \left( \prod^{m_2}_{k=1} T_{2,k}S_{2,k} \right) \cdots \left( \prod^{m_n}_{k=1} T_{n,k}S_{n,k} \right) R_{E(Z_n)}    \right) \\
&= E\left(L_{E(Z_0)} \left( \prod^{m_1}_{k=1} T_{1,k}S_{1,k} \right)L_{E(Z_1)} \left( \prod^{m_2}_{k=1} T_{2,k}S_{2,k} \right) \cdots \left( \prod^{m_n}_{k=1} T_{n,k}S_{n,k} \right) L_{E(Z_n)}    \right)
\end{align*}
as required.
\end{proof}

\begin{exam}
Let $\mathbb{F}_2$ denote the free group on $2$ generators $u_1, u_2$ and let $\varphi$ be the vector state on $\B(\ell_2(\mathbb{F}_2))$ corresponding to the point mass at the identity.  If $\lambda, \rho : \mathbb{F}_2 \to \B(\ell_2(\mathbb{F}_2))$ denote the left and right regular representations respectively, recall $\{(\lambda(u_k), \rho(u_k))\}_{k =1,2}$ are bi-free two-faced families with respect to $\varphi$.  Hence Theorem \ref{thm:monotone-independence-of-algebras} implies that 
\[
\alg \left( \left\{ \lambda(u_1^p) \rho(u_1^q) \, \mid \, p \in \mathbb{N}, q \in \mathbb{Z} \right\} \right)   \qand   \alg \left( \left\{ \lambda(u_2^p) \, \mid \, p \in \mathbb{Z} \right\} \right)
\]
are monotonically independent with respect to $\varphi$.
\end{exam}

\begin{rem}
One can easily modify the example in Remark \ref{rem:operator-model-counterexample} by choosing $T_1, T_2, S_1, S_2$ such that all $(\ell, r)$-cumulants involving these operators are zero except $\kappa_\chi(T_1, T_2, T_2, S_1, S_2)=1$ to construct $C'_1, C'_2, D'_1, D'_2$ such that the correct algebras are classically, freely, monotonically, or Boolean independent yet
\[
\{(\bC 1_\A + \alg(C'_k), \bC 1_\A +\alg(D'_k))\}_{k \in \{1,2\}}
\]
are not bi-free.
\end{rem}

\subsection{Embedding Monotone Independence into Bi-Free Probability}

This section describes a construction for which given any pair of $B$-algebras $A_1$ and $A_2$ that are monotonically independent over $B$ there exist $C'_1, C_2, D'_1, D_2$ as in Theorem \ref{thm:monotone-independence-of-algebras} such that one may view $A_1 \subseteq \alg(C'_1D'_1)$ and $A_2 \subseteq C_2$.

\begin{cons}
\label{cons:monotone2}
Let $(\mathcal{A}, \Phi)$ be a $B$-non-commutative probability space.  Recall as in Construction \ref{cons:Boolean} that we may view $\A$ as a $B$-$B$-module with specified $B$-vector state, which we denote by $(\A, \mathring{\A}, \Phi)$, and $\A \oplus \A$ can be made into a $B$-$B$-module with specified $B$-vector state $(\A \oplus \A, \mathring{\A} \oplus \A, \Psi)$ where
\[
\Psi(Z_1 \oplus Z_2) = \Phi(Z_1).
\]

Let $A_1$ and $A_2$ be $B$-algebras contained in $\A$ that are monotonically independent over $B$ with respect to $\Phi$. Let $(\Y_1, \mathring{\Y}_1, \Psi_1)$ denote $(\A \oplus \A, \mathring{\A} \oplus \A, \Psi)$, let $(\Y_2, \mathring{\Y}_2, \Psi_2)$ denote $(\A, \mathring{\A}, \Phi)$, let $E$ denote the expectation of $\L(\ast_{m \in \{1,2\}}(\Y_m, \mathring{\Y}_m, \Psi_k))$ onto $B$, and let
\begin{align*}
\lambda_k &: \L_\ell(\Y_k) \to \L_\ell(\ast_{m \in \{1,2\}} (\Y_m, \mathring{\Y}_m, \Psi_m)) \quad \text{and} \\
\rho_k &: \L_r(\Y_k) \to \L_r(\ast_{m \in \{1,2\}} (\Y_m, \mathring{\Y}_m, \Psi_m))
\end{align*}
be the left and right regular representations onto the $k^{\text{th}}$ term respectively.  By definition the pairs of $B$-faces $\{(\lambda_k(\L_\ell(\Y_k)), \rho_k(\L_r(\Y_k)))\}$ are bi-free inside the $B$-$B$-non-commutative probability space $\L(\ast_{m \in \{1,2\}} (\Y_m, \mathring{\Y}_m, \Psi_m))$.

Let $C_2 = \lambda_2(\L_\ell(\Y_2))$, let $D_2 = \rho_2(\L_r(\Y_2))$, let $D'_1 = \{\rho_1(S_{1_B})\} \subseteq \rho_1(\L_r(\Y_k))$, and let
\[
C'_1 = \{ \lambda_1(T_Z) \, \mid \, Z \in A_1\} \subseteq \lambda_1(\L_\ell(\Y_1))
\]
where $S_{1_B}$ and $T_Z$ were as defined in Construction \ref{cons:Boolean}.  It is clear by Construction \ref{cons:Boolean} that $C'_1, C_2, D'_1, D_2$ satisfy the assumptions of Theorem \ref{thm:monotone-independence-of-algebras}.  Furthermore, if one considers the unital homomorphism  $\beta_2 : A_2 \to C_2$ defined by 
\[
\beta_2(Z) = \lambda_2(Z)
\]
where one views $Z \in \A_\ell$ for all $Z \in A_2$ by left multiplication, and if one considers the linear map $\beta_1 : A_1 \to \alg(C'_1D'_1)$ defined by
\[
\beta_1(Z) = \lambda_1(T_Z)\rho_1(S_{1_B})
\]
for all $Z \in A_1$, Theorem \ref{thm:monotone-independence-of-algebras} implies the joint distributions of elements of $A_1$ and $A_2$ with respect to $\Phi$ are equal to the joint distributions of their images under $\beta_1$ and $\beta_2$ respectively with respect to $E$.
\end{cons}

\subsection{Monotone Bi-Non-Crossing Partitions}

To develop the analogue of equations (\ref{eq:free-via-lr-cumulants}) and (\ref{eq:independent-via-LR-cumulants}) of Section \ref{sec:ClassicalAndFreeIndependence} for monotone independence, we need to restrict ourselves to specific bi-non-crossing partitions.
\begin{defn}
Let $\chi : \{1,\ldots, n\}\to \{\ell,r\}$.  A bi-non-crossing partition $\pi \in BNC(\chi)$ is said to be \emph{monotone} if $\pi \in \bncs(\chi)$ and whenever $m, p, q \in \{1,\ldots, n\}$ are such that $m < p < q$, $\chi(p) = r$, $\chi(m) = \chi(q) = \ell$, then $m$ and $q$ are not in the same block of $\pi$.  The set of monotone bi-non-crossing partitions is denoted by $\bncm(\chi)$.
\end{defn}

\begin{exam}
For $\chi : \{1,\ldots, 6\} \to \{\ell, r\}$ with $\chi^{-1}(\{\ell\}) = \{2,3,4,6\}$ and $\chi^{-1}(\{r\}) = \{1,5\}$, the elements of $\bncm(\chi)$ may be represented via the following bi-non-crossing diagrams.
\begin{align*}
	\begin{tikzpicture}[baseline]
			\draw[thick,dashed] (-.5,2.75) -- (-.5,-.25) -- (.5,-.25) -- (.5,2.75);
			\node[right] at (.5,2.5) {1};
			\draw[black,fill=black] (.5,2.5) circle (0.05);
			\node[left] at (-.5,2) {2};
			\draw[black,fill=black] (-.5,2) circle (0.05);
			\node[left] at (-.5,1.5) {3};
			\draw[black,fill=black] (-.5,1.5) circle (0.05);
			\node[left] at (-.5,1) {4};
			\draw[black,fill=black] (-.5,1) circle (0.05);
			\node[right] at (.5,.5) {5};
			\draw[black,fill=black] (.5,.5) circle (0.05);
			\node[left] at (-.5,0) {6};
			\draw[black,fill=black] (-.5,0) circle (0.05);
	\end{tikzpicture}	
	\quad
	\begin{tikzpicture}[baseline]
			\draw[thick,dashed] (-.5,2.75) -- (-.5,-.25) -- (.5,-.25) -- (.5,2.75);
			\node[right] at (.5,2.5) {1};
			\draw[black,fill=black] (.5,2.5) circle (0.05);
			\node[left] at (-.5,2) {2};
			\draw[black,fill=black] (-.5,2) circle (0.05);
			\node[left] at (-.5,1.5) {3};
			\draw[black,fill=black] (-.5,1.5) circle (0.05);
			\node[left] at (-.5,1) {4};
			\draw[black,fill=black] (-.5,1) circle (0.05);
			\node[right] at (.5,.5) {5};
			\draw[black,fill=black] (.5,.5) circle (0.05);
			\node[left] at (-.5,0) {6};
			\draw[black,fill=black] (-.5,0) circle (0.05);
			\draw[thick, black] (-.5,2) -- (-.125,2) -- (-.125, 1.5) -- (-.5,1.5);
	\end{tikzpicture}	
	\quad
	\begin{tikzpicture}[baseline]
			\draw[thick,dashed] (-.5,2.75) -- (-.5,-.25) -- (.5,-.25) -- (.5,2.75);
			\node[right] at (.5,2.5) {1};
			\draw[black,fill=black] (.5,2.5) circle (0.05);
			\node[left] at (-.5,2) {2};
			\draw[black,fill=black] (-.5,2) circle (0.05);
			\node[left] at (-.5,1.5) {3};
			\draw[black,fill=black] (-.5,1.5) circle (0.05);
			\node[left] at (-.5,1) {4};
			\draw[black,fill=black] (-.5,1) circle (0.05);
			\node[right] at (.5,.5) {5};
			\draw[black,fill=black] (.5,.5) circle (0.05);
			\node[left] at (-.5,0) {6};
			\draw[black,fill=black] (-.5,0) circle (0.05);
			\draw[thick, black] (-.5,1.5) -- (-.125,1.5) -- (-.125, 1) -- (-.5,1);
	\end{tikzpicture}	
	\quad 
	\begin{tikzpicture}[baseline]
			\draw[thick,dashed] (-.5,2.75) -- (-.5,-.25) -- (.5,-.25) -- (.5,2.75);
			\node[right] at (.5,2.5) {1};
			\draw[black,fill=black] (.5,2.5) circle (0.05);
			\node[left] at (-.5,2) {2};
			\draw[black,fill=black] (-.5,2) circle (0.05);
			\node[left] at (-.5,1.5) {3};
			\draw[black,fill=black] (-.5,1.5) circle (0.05);
			\node[left] at (-.5,1) {4};
			\draw[black,fill=black] (-.5,1) circle (0.05);
			\node[right] at (.5,.5) {5};
			\draw[black,fill=black] (.5,.5) circle (0.05);
			\node[left] at (-.5,0) {6};
			\draw[black,fill=black] (-.5,0) circle (0.05);
			\draw[thick, black] (-.5,2) -- (-.125,2) -- (-.125, 1) -- (-.5,1);
	\end{tikzpicture}	
	\quad 
	\begin{tikzpicture}[baseline]
			\draw[thick,dashed] (-.5,2.75) -- (-.5,-.25) -- (.5,-.25) -- (.5,2.75);
			\node[right] at (.5,2.5) {1};
			\draw[black,fill=black] (.5,2.5) circle (0.05);
			\node[left] at (-.5,2) {2};
			\draw[black,fill=black] (-.5,2) circle (0.05);
			\node[left] at (-.5,1.5) {3};
			\draw[black,fill=black] (-.5,1.5) circle (0.05);
			\node[left] at (-.5,1) {4};
			\draw[black,fill=black] (-.5,1) circle (0.05);
			\node[right] at (.5,.5) {5};
			\draw[black,fill=black] (.5,.5) circle (0.05);
			\node[left] at (-.5,0) {6};
			\draw[black,fill=black] (-.5,0) circle (0.05);
			\draw[thick, black] (-.5,2) -- (-.125,2) -- (-.125, 1) -- (-.5,1);
			\draw[thick, black] (-.5,1.5) -- (-.125,1.5);
	\end{tikzpicture}	\\ \\
	\begin{tikzpicture}[baseline]
			\draw[thick,dashed] (-.5,2.75) -- (-.5,-.25) -- (.5,-.25) -- (.5,2.75);
			\node[right] at (.5,2.5) {1};
			\draw[black,fill=black] (.5,2.5) circle (0.05);
			\node[left] at (-.5,2) {2};
			\draw[black,fill=black] (-.5,2) circle (0.05);
			\node[left] at (-.5,1.5) {3};
			\draw[black,fill=black] (-.5,1.5) circle (0.05);
			\node[left] at (-.5,1) {4};
			\draw[black,fill=black] (-.5,1) circle (0.05);
			\node[right] at (.5,.5) {5};
			\draw[black,fill=black] (.5,.5) circle (0.05);
			\node[left] at (-.5,0) {6};
			\draw[black,fill=black] (-.5,0) circle (0.05);
			\draw[thick, black] (.5,.5) -- (0.125,.5) -- (0.125, 2.5) -- (.5,2.5);
	\end{tikzpicture}	
	\quad
	\begin{tikzpicture}[baseline]
			\draw[thick,dashed] (-.5,2.75) -- (-.5,-.25) -- (.5,-.25) -- (.5,2.75);
			\node[right] at (.5,2.5) {1};
			\draw[black,fill=black] (.5,2.5) circle (0.05);
			\node[left] at (-.5,2) {2};
			\draw[black,fill=black] (-.5,2) circle (0.05);
			\node[left] at (-.5,1.5) {3};
			\draw[black,fill=black] (-.5,1.5) circle (0.05);
			\node[left] at (-.5,1) {4};
			\draw[black,fill=black] (-.5,1) circle (0.05);
			\node[right] at (.5,.5) {5};
			\draw[black,fill=black] (.5,.5) circle (0.05);
			\node[left] at (-.5,0) {6};
			\draw[black,fill=black] (-.5,0) circle (0.05);
			\draw[thick, black] (-.5,2) -- (-.125,2) -- (-.125, 1.5) -- (-.5,1.5);
			\draw[thick, black] (.5,.5) -- (0.125,.5) -- (0.125, 2.5) -- (.5,2.5);
	\end{tikzpicture}	
	\quad
	\begin{tikzpicture}[baseline]
			\draw[thick,dashed] (-.5,2.75) -- (-.5,-.25) -- (.5,-.25) -- (.5,2.75);
			\node[right] at (.5,2.5) {1};
			\draw[black,fill=black] (.5,2.5) circle (0.05);
			\node[left] at (-.5,2) {2};
			\draw[black,fill=black] (-.5,2) circle (0.05);
			\node[left] at (-.5,1.5) {3};
			\draw[black,fill=black] (-.5,1.5) circle (0.05);
			\node[left] at (-.5,1) {4};
			\draw[black,fill=black] (-.5,1) circle (0.05);
			\node[right] at (.5,.5) {5};
			\draw[black,fill=black] (.5,.5) circle (0.05);
			\node[left] at (-.5,0) {6};
			\draw[black,fill=black] (-.5,0) circle (0.05);
			\draw[thick, black] (-.5,1.5) -- (-.125,1.5) -- (-.125, 1) -- (-.5,1);
			\draw[thick, black] (.5,.5) -- (0.125,.5) -- (0.125, 2.5) -- (.5,2.5);
	\end{tikzpicture}	
	\quad 
	\begin{tikzpicture}[baseline]
			\draw[thick,dashed] (-.5,2.75) -- (-.5,-.25) -- (.5,-.25) -- (.5,2.75);
			\node[right] at (.5,2.5) {1};
			\draw[black,fill=black] (.5,2.5) circle (0.05);
			\node[left] at (-.5,2) {2};
			\draw[black,fill=black] (-.5,2) circle (0.05);
			\node[left] at (-.5,1.5) {3};
			\draw[black,fill=black] (-.5,1.5) circle (0.05);
			\node[left] at (-.5,1) {4};
			\draw[black,fill=black] (-.5,1) circle (0.05);
			\node[right] at (.5,.5) {5};
			\draw[black,fill=black] (.5,.5) circle (0.05);
			\node[left] at (-.5,0) {6};
			\draw[black,fill=black] (-.5,0) circle (0.05);
			\draw[thick, black] (-.5,2) -- (-.125,2) -- (-.125, 1) -- (-.5,1);
			\draw[thick, black] (.5,.5) -- (0.125,.5) -- (0.125, 2.5) -- (.5,2.5);
	\end{tikzpicture}	
	\quad 
	\begin{tikzpicture}[baseline]
			\draw[thick,dashed] (-.5,2.75) -- (-.5,-.25) -- (.5,-.25) -- (.5,2.75);
			\node[right] at (.5,2.5) {1};
			\draw[black,fill=black] (.5,2.5) circle (0.05);
			\node[left] at (-.5,2) {2};
			\draw[black,fill=black] (-.5,2) circle (0.05);
			\node[left] at (-.5,1.5) {3};
			\draw[black,fill=black] (-.5,1.5) circle (0.05);
			\node[left] at (-.5,1) {4};
			\draw[black,fill=black] (-.5,1) circle (0.05);
			\node[right] at (.5,.5) {5};
			\draw[black,fill=black] (.5,.5) circle (0.05);
			\node[left] at (-.5,0) {6};
			\draw[black,fill=black] (-.5,0) circle (0.05);
			\draw[thick, black] (-.5,2) -- (-.125,2) -- (-.125, 1) -- (-.5,1);
			\draw[thick, black] (-.5,1.5) -- (-.125,1.5);
			\draw[thick, black] (.5,.5) -- (0.125,.5) -- (0.125, 2.5) -- (.5,2.5);
	\end{tikzpicture}	
\end{align*}
\end{exam}
\begin{rem}
For $\chi : \{1,\ldots, n\} \to \{\ell, r\}$, let $\overline{\chi} : \{1,\ldots, n\} \to \{\ell, r\}$ denote the constant function $\overline{\chi}(k) = \ell$.  If $\pi \in \bncm(\chi)$, then $\pi$ is naturally a non-crossing partition on $\{1,\ldots, n\}$ and corresponds to a unique element $\overline{\pi}$ of $BNC(\overline{\chi})$.  Let $\overline{\bncm(\chi)}$ to denote the image in $BNC(\overline{\chi})$ under this map.
\end{rem}

\begin{thm}
Let $(\A, \Phi)$ be a $B$-non-commutative probability space, let $A_1$ and $A_2$ be $B$-algebras contained in $\A$, and let $\Psi$ be the conditional expectation on $A_1 \ast_B A_2$ determined by $\Phi$ for which $A_1$ and $A_2$ are monotonically independent over $B$.  Viewing $(\A, \Phi)$ as a $B$-$B$-non-commutative probability space via Remark \ref{rem:B-ncps-from-B-B-ncps}, for all $\epsilon : \{1,\ldots, n\} \to \{1,2\}$ and for all $Z_k \in A_{\epsilon(k)}$, 
\begin{align}
\Psi(Z_1 \cdots Z_n) = \sum_{\overline{\pi} \in \overline{\bncm(\chi_\epsilon)}} \kappa_{\overline{\pi}}(Z_1, \ldots, Z_n) \label{eq:monotone-independent-via-LR-cumualnts}
\end{align}
where $\chi_\epsilon : \{1, \ldots, n\} \to \{\ell, r\}$ is defined by
\[
\chi_\epsilon(k) = \left\{
\begin{array}{ll}
\ell & \mbox{if } \epsilon(k) = 2  \\
r & \mbox{if } \epsilon(k) = 1
\end{array} \right. .
\]
\end{thm}
\begin{proof}
Fix $\epsilon : \{1,\ldots, n\} \to \{1,2\}$ and $Z_k \in A_{\epsilon(k)}$.  Let $W = \epsilon^{-1}(\{1\}) = \{q_1 < q_2 < \cdots < q_m\}$, $q_0 = 0$, $q_{m+1} = n+1$, and $V_k = \{q_{k-1} + 1,\ldots, q_{k} - 1\}$ for all $k \in \{1,\ldots, m+1\}$.   Since each $V_k$ is a $\overline{\chi_\epsilon}$-interval, if $\chi_0 : \{1,\ldots, m\} \to \{\ell, r\}$ is the constant map $\chi_0(k) = \ell$, then the properties of bi-multiplicative functions imply that
\begin{align*}
& \sum_{\overline{\pi} \in \overline{\bncm(\chi_\epsilon)}} \kappa_{\overline{\pi}}(Z_1, \ldots, Z_n) \\
&= \sum_{\sigma \in BNC(\chi_0)}  E(Z_{q_{0} + 1}\cdots Z_{q_{1} - 1}) \kappa_{\sigma }(Z_{q_1} L_{E(Z_{q_{1} + 1}\cdots Z_{q_{2} - 1})}, \ldots, Z_{q_m}L_{E(Z_{q_{m} + 1}\cdots Z_{q_{m+1} - 1})} ) \\
&=   E(Z_{q_{0} + 1}\cdots Z_{q_{1} - 1}) \E_{1_{\chi_0}}\left(Z_{q_1} L_{E(Z_{q_{1} + 1}\cdots Z_{q_{2} - 1})}, \ldots, Z_{q_m}L_{E(Z_{q_{m} + 1}\cdots Z_{q_{m+1} - 1})} \right) \\
&= E\left(L_{E(Z_{q_{0} + 1}\cdots Z_{q_{1} - 1})}Z_{q_1} L_{E(Z_{q_{1} + 1}\cdots Z_{q_{2} - 1})}Z_{q_2} \cdots Z_{q_m}L_{E(Z_{q_{m} + 1}\cdots Z_{q_{m+1} - 1})}  \right) \\
&=\Psi(Z_1 \cdots Z_n). \qedhere
\end{align*}
\end{proof}

\section{Bi-Freeness when Tensoring with $M_n(\mathbb{C})$}
\label{sec:Tensoring}

This section demonstrates how matrices of bi-freely independent algebras over $B$ are bi-freely independent over $M_n(B)$.

\subsection{Free Independence over $\boldsymbol{M_n(B)}$}

Let $(\A, \Phi)$ be a $B$-non-commutative probability space.  It is elementary to show that if $\Phi_n : M_n(\A) \to M_n(B)$ is the linear map defined by
\[
\Phi_n([Z_{i,j}]) = [\Phi(Z_{i,j})]
\]
for all $[Z_{i,j}] \in M_n(\A)$, then $(M_n(\A), \Phi_n)$ is a $M_n(B)$-non-commutative probability space.  Furthermore, it is well-known that if $\{A_k\}_{k \in K}$ are unital algebras of $\A$ containing $B$ that are freely independent over $B$, then $\{M_n(A_k)\}_{k \in K}$ are freely independent algebras over $M_n(B)$ in $(M_n(\A), \Phi_n)$.   A proof of this result can easily be obtained using the fact that free independence is equivalent to the moment of any alternating, centred product being zero; a characterization that is non-existent for bi-free independence.

\subsection{Matrices of $\boldsymbol{B}$-$\boldsymbol{B}$-Non-Commutative Probability Spaces}

To proceed with the bi-free analogue of the above result, we need to understand how to interpret $(M_n(C), M_n(D))$ for given a pair of $B$-faces $(C, D)$.
\begin{rem}
Given a $B$-$B$-non-commutative probability space $(\A, E_\A, \epsilon)$, the main issue with trying to make $M_n(\A)$ an $M_n(B)$-$M_n(B)$-non-commutative probability space with the expectation
\[
E_{M_n(\A)}[Z_{i,j}] = [E_\A(Z_{i,j})]
\]
is the construction of a unital homomorphism $\epsilon_n : M_n(B) \otimes M_n(B)^{\mathrm{op}} \to M_n(\A)$.  In the case that $B = \bC$, the range of $\epsilon_n$ would need to include $M_n(\bC) \subseteq M_n(\A)$ which would imply the left and right algebras $M_n(\A)_\ell$ and $M_n(\A)_r$ are trivial (i.e. $\A I_n$) and the discussion of bi-free independence moot.

The main issue is that $M_n(\A)$ does not correctly distinguish left and right operators.  In the scalar setting, with a non-commutative probability space $(\A, \varphi)$, any element of $\A$ is permitted to be a left or a right operator, yet, to consider operator-valued bi-free independence, only certain operators are allowed to be left or right operators.
\end{rem}

The procedure described below takes a $B$-$B$-non-commutative probability space $(\A, E_\A, \epsilon)$ and constructs an associated $M_n(B)$-$M_n(B)$-non-commutative probability space for which $M_n(\A_\ell)$ and $M_n(\A_r)$ identify with left and right operators respectively.
\begin{cons}
\label{cons:Mn-of-a-B-B-ncps}
Let $(\A, E_\A, \varepsilon)$ be a $B$-$B$-non-commutative probability space.  Recall Theorem \ref{thm:representing-bb-ncps} implies there exists a $B$-$B$-bimodule with specified $B$-vector state $(\X, \mathring{\X}, p)$ so that we may view $\A \subseteq \L(\X)$ and
\[
E_\A(Z) = p(Z (1_B))
\]
for all $Z \in \A$.  

To make $M_n(\X)$ an $M_n(B)$-$M_n(B)$-bimodule, for $[b_{i,j}] \in M_n(B)$ and $[\xi_{i,j}] \in M_n(\X)$, the left and right actions of $[b_{i,j}]$ on $[\xi_{i,j}]$ are defined as
\begin{align*}
L_{[b_{i,j}]}([\xi_{i,j}]) &= [b_{i,j}] \cdot [\xi_{i,j}] = \left[ \sum^n_{k=1} b_{i,k} \cdot \xi_{k,j} \right]      = \left[ \sum^n_{k=1} L_{b_{i,k}} (\xi_{k,j}) \right] \text{ and} \\
R_{[b_{i,j}]}([\xi_{i,j}]) &= [\xi_{i,j}] \cdot [b_{i,j}] = \left[ \sum^n_{k=1} \xi_{i,k} \cdot b_{k,j} \right] = \left[ \sum^n_{k=1} R_{b_{k,j}} (\xi_{i,k}) \right].
\end{align*}
Simple computations verify that these actions indeed make $M_n(\X)$ an $M_n(B)$-$M_n(B)$-bimodule (i.e. $L_{[b_{i,j}]}$ and $R_{[b'_{i,j}]}$ commute, $[b_{i,j}] \to L_{[b_{i,j}]}$ is a homomorphism, and $[b_{i,j}] \to R_{[b_{i,j}]}$ is an anti-homomorphism).  Furthermore, since
\[
M_n(\X) = M_n(B) \oplus M_n(\mathring{\X})
\]
as direct sum of $M_n(B)$-$M_n(B)$-bimodules, $(M_n(\X), M_n(\mathring{\X}), p_n)$ is an $M_n(B)$-$M_n(B)$-bimodule with specified $M_n(B)$-vector state where $p_n : M_n(\X) \to M_n(B)$ is defined by
\[
p_n([\xi_{i,j}]) = [p(\xi_{i,j})].
\]
Therefore $\L(M_n(\X))$ is an $M_n(B)$-$M_n(B)$-non-commutative probability space with conditional expectation $E_{\L(M_n(\X))} : \L(M_n(\X)) \to M_n(B)$ defined for all $Z \in \L(M_n(\X))$ by
\[
E_{\L(M_n(\X))}(Z) = p_n(Z (I_{n,B}))
\]
where $I_{n,B} \in M_n(B)$ is the unit of $M_n(B)$.  

Given a matrix $[Z_{i,j}] \in M_n(\A_\ell)$, we may view $[Z_{i,j}] \in \L(M_n(\X))$ by defining
\[
[Z_{i,j}]( [\xi_{i,j}]) = \left[  \sum^n_{k=1} Z_{i,k} (\xi_{k,j})   \right]
\]
for all $[\xi_{i,j}] \in M_n(\X)$.  Note $[L_{b_{i,j}}] = L_{[b_{i,j}]}$ and $M_n(\A_\ell)$ commutes with $\{R_{[b_{i,j}]} \, \mid \, [b_{i,j}] \in M_n(B)\}$ under this representation.  Thus we may view
\[
\{L_{[b_{i,j}]} \, \mid \, [b_{i,j}] \in M_n(B)\} \subseteq M_n(\A_\ell) \subseteq \L(M_n(\X))_\ell.
\]
Moreover, note if $[Z_{i,j}] \in M_n(\A_\ell)$, then, under this identification,
\[
E_{\L(M_n(\X))}([Z_{i,j}]) = p_n([Z_{i,j} (1_B)]) = [p(Z_{i,j} (1_B))] = [E_\A(Z_{i,j})].
\]

Similarly, given a matrix $[Z_{i,j}] \in M_n(\A_r^{\mathrm{op}})^{\mathrm{op}}$, we may view $[Z_{i,j}] \in \L(M_n(\X))$ by defining
\[
[Z_{i,j}]( [\xi_{i,j}]) = \left[  \sum^n_{k=1} Z_{k,j} (\xi_{i,k})   \right]
\]
for all $[\xi_{i,j}] \in M_n(\X)$.  Note $[R_{b_{i,j}}] = R_{[b_{i,j}]}$ and $M_n(\A_r^{\mathrm{op}})^{\mathrm{op}}$ commutes with $\{L_{[b_{i,j}]} \, \mid \, [b_{i,j}] \in M_n(B)\}$ under this representation.  Thus we may view
\[
\{R_{[b_{i,j}]} \, \mid \, [b_{i,j}] \in M_n(B)\} \subseteq M_n(\A_r^{\mathrm{op}})^{\mathrm{op}} \subseteq \L(M_n(\X))_r.
\]
Moreover, note if $[Z_{i,j}] \in M_n(\A_r^{\mathrm{op}})^{\mathrm{op}}$, then, under this identification,
\[
E_{\L(M_n(\X))}([Z_{i,j}]) = p_n([Z_{i,j} (1_B)]) = [p(Z_{i,j} (1_B))] = [E_\A(Z_{i,j})].
\]

Hence, given pairs of $B$-faces $\{(C_k, D_k)\}_{k \in K}$ in a $B$-$B$-non-commutative probability space $\A$, we can consider the pairs of $M_n(B)$-faces $\{(M_n(C_k), M_n(D_k^{\mathrm{op}})^{\mathrm{op}})\}_{k \in K}$ inside the $M_n(B)$-$M_n(B)$-non-commutative probability space $\L(M_n(\X))$.
\end{cons}
\begin{rem}
One may be slightly concerned by the necessity of using $\X$.  However, Theorem \ref{thm:representing-bb-ncps} constructs $\X$ as a quotient of $\A$ and $\X$ is directly associated to $\A$.  In fact, in the case $B = \bC$, $\X = \A$ in which case we are considering the $M_n(\bC)$-$M_n(\bC)$-non-commutative probability space $\L(M_n(\A))$, elements of $M_n(\A)$ acting as left operators on $M_n(\A)$ by left matrix multiplication, and elements of $M_n(\A)$ acting as right operators on $M_n(\A)$ by right multiplication coupled with the opposite action of $\A$.
\end{rem}
\begin{rem}
One can verify via Remark \ref{rem:B-ncps-from-B-B-ncps} that Construction \ref{cons:Mn-of-a-B-B-ncps} is an alternate way of viewing $M_n(\A)$ when $\A$ is a $B$-non-commutative probability space.
\end{rem}

\subsection{Bi-Free Independence over $\boldsymbol{M_n(B)}$}

Using Construction \ref{cons:Mn-of-a-B-B-ncps}, we may now state the main result of this section.
\begin{thm}
\label{thm:bi-free-preserved-when-tensored-with-Mn}
Let $(\A, E, \varepsilon)$ be a $B$-$B$-non-commutative probability space and let $\{(C_k, D_k)\}_{k \in K}$ be bi-free pairs of $B$-faces of $\A$.  Then the pairs of $M_n(B)$-faces $\{(M_n(C_k), M_n(D_k^{\mathrm{op}})^{\mathrm{op}})\}_{k \in K}$ are bi-freely independent over $M_n(B)$.
\end{thm}
To proceed with the proof of Theorem \ref{thm:bi-free-preserved-when-tensored-with-Mn}, let $\{F_{i,j}\}_{i,j=1}^n$ denote the canonical matrix units of $M_n(\bC)$ (so $F_{i,k}F_{m,j} = \delta_{k,m} F_{i,j}$).  In addition, $Z \otimes F_{i,j}$ represents the $n$ by $n$ matrix with $Z$ in the $(i,j)^{\mathrm{th}}$ position and zeros elsewhere and $Z \otimes 0 = 0$.  The following technical lemma will enable the proof of Theorem \ref{thm:bi-free-preserved-when-tensored-with-Mn}.
\begin{lem}
\label{lem:tensor-technical-lemma}
Let $(\A, E, \varepsilon)$ be a $B$-$B$-non-commutative probability space and let $\L(M_n(\X))$ be as described in Construction \ref{cons:Mn-of-a-B-B-ncps}.   For all $\chi : \{1,\ldots, q\} \to \{\ell, r\}$, for all
\[
Z_k \in \left\{
\begin{array}{ll}
\A_\ell & \mbox{if } \chi(k) = \ell  \\
\A_r & \mbox{if } \chi(k) = r
\end{array} \right.,
\]
for all $\{i_k\}^q_{k=1}, \{j_k\}^q_{k=1} \subseteq \{1,\ldots, n\}$,  and for all $\pi \in BNC(\chi)$, 
\[
\E_{\pi}(Z_1 \otimes F_{i_1, j_1}, \ldots, Z_q \otimes F_{i_q, j_q}) = \E_{\pi}(Z_1, \ldots, Z_q) \otimes F_\chi((i_1, \ldots, i_q), (j_1, \ldots, j_q))
\]
where the operator-valued bi-free moment functions are computed in the appropriate spaces and
\[
F_\chi((i_1, \ldots, i_q), (j_1, \ldots, j_q)) = F_{i_{s_\chi(1)}, j_{s_\chi(1)}} \cdots F_{i_{s_\chi(q)}, j_{s_\chi(q)}} \in M_n(\bC).
\]
\end{lem}
\begin{proof}
The proof proceeds by induction on the number of blocks of $\pi$.  For the base case, suppose $\pi$ has precisely one block; that is, $\pi = 1_\chi$.  Using the operations in Construction \ref{cons:Mn-of-a-B-B-ncps}, it is elementary to verify that
\[
(Z_1 \otimes F_{i_1, j_1}) \cdots (Z_q \otimes F_{i_q, j_q}) I_{n,B} = (Z_1 \cdots Z_q (1_B)) \otimes F_\chi((i_1, \ldots, i_q), (j_1, \ldots, j_q))
\]
(where $Z_k \otimes F_{i_k, j_k}$ acts as an element of $\mathcal{M}_n(\A_{\chi(k)})$) so that
\begin{align*}
\E_{1_\chi}(Z_1 \otimes F_{i_1, j_1}, \ldots, Z_q \otimes F_{i_q, j_q}) &= E_{\L(M_n(\X))}((Z_1 \otimes F_{i_1, j_1}) \cdots (Z_q \otimes F_{i_q, j_q})) \\
&= p_n((Z_1 \otimes F_{i_1, j_1}) \cdots (Z_q \otimes F_{i_q, j_q}) I_{n,B})\\
&= p(Z_1 \cdots Z_q (1_B)) \otimes F_\chi((i_1, \ldots, i_q), (j_1, \ldots, j_q)) \\
&= E_\A(Z_1 \cdots Z_q) \otimes F_\chi((i_1, \ldots, i_q), (j_1, \ldots, j_q)) \\
&= \E_{1_\chi}(Z_1, \ldots, Z_q) \otimes F_\chi((i_1, \ldots, i_q), (j_1, \ldots, j_q)).
\end{align*}

To proceed inductively, fix $\pi \in BNC(\chi)$ and suppose the result holds for all bi-non-crossing partitions with fewer blocks than $\pi$.  Let $W$ be any block of $\pi$ that is a $\chi$-interval.  By the base case
\begin{align*}
\E_{\pi|_{W}}(&(Z_1 \otimes F_{i_1, j_1}, \ldots, Z_q \otimes F_{i_q, j_q})|_{W}) \\
&= \E_{\pi|_{W}}((Z_1, \ldots, Z_q)|_{W}) \otimes F_{\chi|_W}((i_1, \ldots, i_q)|_W, (j_1, \ldots, j_q)|_W).
\end{align*}

For simplicity of notation, let 
\begin{align*}
b = \E_{\pi|_{W}}((Z_1, \ldots, Z_q)|_{W}) \qqand G = F_{\chi|_W}((i_1, \ldots, i_q)|_W, (j_1, \ldots, j_q)|_W).
\end{align*}
Further, let
\[
p = \max_{\prec_\chi}\left(\left\{k \in \{1,\ldots, q\} \, \mid \, k \prec_\chi \min_{\prec_\chi}(W)\right\}\right).
\]
The proof is divided into two cases based on $p$.

\underline{\textit{Case 1: $p \neq -\infty$.}}   
Notice if $F_{i_p, j_p}G = \delta F_{i,j}$ where $\delta \in \{0,1\}$, then it is clear by the definition of $\prec_\chi$ that
\begin{align*}
\delta F_{\chi|_{W^c}}((i_1, \ldots, i_{p-1}, & i, i_{p+1}, \ldots, i_q)|_{W^c}, (j_1, \ldots, j_{p-1}, j, j_{p+1}, \ldots, j_q)|_{W^c}) \\
&  = F_\chi((i_1, \ldots, i_q), (j_1, \ldots, j_q)).
\end{align*}
We need to further divide the discussion into two cases. 

If $\chi(p) = \ell$, notice the actions in Construction \ref{cons:Mn-of-a-B-B-ncps} imply
\[
(Z_p \otimes F_{i_p, j_p}) L_{  b \otimes G } = Z_p L_b \otimes F_{i_p, j_p}G.
\]
Hence, by the properties of bi-multiplicative functions  and the induction hypothesis, 
\begin{align*}
\E_{\pi}&(Z_1 \otimes F_{i_1, j_1}, \ldots, Z_q \otimes F_{i_q, j_q}) \\
&= \E_{\pi|_{W^c}}((Z_1 \otimes F_{i_1, j_1}, \ldots,   (Z_p \otimes F_{i_p, j_p}) L_{  b \otimes G }, \ldots,   Z_q \otimes F_{i_q, j_q})|_{W^c})\\
&= \E_{\pi|_{W^c}}((Z_1 \otimes F_{i_1, j_1}, \ldots,   (Z_pL_b \otimes F_{i_p, j_p}G ), \ldots,   Z_q \otimes F_{i_q, j_q})|_{W^c})\\
&= \E_{\pi|_{W^c}}((Z_1, \ldots,   Z_p L_b, \ldots,   Z_q)|_{W^c}) \otimes F_\chi((i_1, \ldots, i_q), (j_1, \ldots, j_q)) \\
&= \E_{\pi}(Z_1, \ldots, Z_q) \otimes F_\chi((i_1, \ldots, i_q), (j_1, \ldots, j_q)).
\end{align*}
Otherwise, if $\chi(p) = r$, notice the actions in Construction \ref{cons:Mn-of-a-B-B-ncps} imply
\[
R_{  b \otimes G } (Z_p \otimes F_{i_p, j_p})  = R_bZ_p \otimes F_{i_p, j_p}G.
\]
Hence, by the properties of bi-multiplicative functions and the induction hypothesis, 
\begin{align*}
\E_{\pi}&(Z_1 \otimes F_{i_1, j_1}, \ldots, Z_q \otimes F_{i_q, j_q}) \\
&= \E_{\pi|_{W^c}}((Z_1 \otimes F_{i_1, j_1}, \ldots,   R_{  b \otimes G } (Z_p \otimes F_{i_p, j_p}), \ldots,   Z_q \otimes F_{i_q, j_q})|_{W^c})\\
&= \E_{\pi|_{W^c}}((Z_1 \otimes F_{i_1, j_1}, \ldots,   (R_bZ_p \otimes F_{i_p, j_p}G ), \ldots,   Z_q \otimes F_{i_q, j_q})|_{W^c})\\
&= \E_{\pi|_{W^c}}((Z_1, \ldots,  R_b Z_p, \ldots,   Z_q)|_{W^c}) \otimes F_\chi((i_1, \ldots, i_q), (j_1, \ldots, j_q)) \\
&= \E_{\pi}(Z_1, \ldots, Z_q) \otimes F_\chi((i_1, \ldots, i_q), (j_1, \ldots, j_q)).
\end{align*}

\underline{\textit{Case 2: $p = -\infty$.}}   It is clear by the definition of $\prec_\chi$ that
\[
GF_{\chi|_{W^c}}((i_1, \ldots, i_q)|_{W^c}, (j_1, \ldots, j_q)|_{W^c}) = F_\chi((i_1, \ldots, i_q), (j_1, \ldots, j_q)).
\]
Hence, by the properties of bi-multiplicative functions and the induction hypothesis, 
\begin{align*}
\E_{\pi}&(Z_1 \otimes F_{i_1, j_1}, \ldots, Z_q \otimes F_{i_q, j_q}) \\
&= (b \otimes G) \E_{\pi|_{W^c}}((Z_1 \otimes F_{i_1, j_1}, \ldots,   Z_q \otimes F_{i_q, j_q})|_{W^c})\\
&= (b \otimes G) \left(\E_{\pi|_{W^c}}((Z_1, \ldots, Z_q)|_{W^c}) \otimes F_{\chi|_W}((i_1, \ldots, i_q)|_{W^c}, (j_1, \ldots, j_q)|_{W^c})\right) \\
&= (b\E_{\pi|_{W^c}}((Z_1, \ldots, Z_q)|_{W^c}))  \otimes (GF_{\chi|_W}((i_1, \ldots, i_q)|_{W^c}, (j_1, \ldots, j_q)|_{W^c}))) \\
&= \E_{\pi}(Z_1, \ldots, Z_q) \otimes F_\chi((i_1, \ldots, i_q), (j_1, \ldots, j_q)).
\end{align*}

As all cases have been covered, the inductive step and thus proof are complete.
\end{proof}
\begin{proof}[Proof of Theorem \ref{thm:bi-free-preserved-when-tensored-with-Mn}]
To prove Theorem \ref{thm:bi-free-preserved-when-tensored-with-Mn} it suffices to verify equation (\ref{eq:universal-polys}) in Theorem \ref{thm:bifree-classifying-theorem}.  In addition, due to linearity, it suffices to verify equation (\ref{eq:universal-polys}) for elements of the form $Z \otimes F_{i,j}$ where $Z$ is an element of a $C_k$ or a $D_k$ and $i,j \in \{1,\ldots, n\}$ are arbitrary.

Let $\chi : \{1,\ldots, q\} \to \{\ell, r\}$, let $\epsilon : \{1, \ldots, q\} \to K$, let
\[
Z_k \in \left\{
\begin{array}{ll}
C_{\epsilon(k)} & \mbox{if } \chi(k) = \ell  \\
D_{\epsilon(k)}  & \mbox{if } \chi(k) = r
\end{array} \right.,
\]
and let $\{i_k\}^q_{k=1}, \{j_k\}^q_{k=1} \subseteq \{1,\ldots, n\}$.  Then, by Lemma \ref{lem:tensor-technical-lemma} and the fact that equation (\ref{eq:universal-polys}) of Theorem \ref{thm:bifree-classifying-theorem} holds for $\{(C_k, D_k)\}_{k \in K}$,
\begin{align*}
&E_{\L(M_n(\X))}((Z_1\otimes F_{i_1, j_1}) \cdots (Z_q \otimes F_{i_q, j_q})) \\
&= E_\A(Z_1 \cdots Z_q) \otimes F_\chi((i_1, \ldots, i_q), (j_1, \ldots, j_q)) \\
&= \sum_{\pi \in BNC(\chi)} \left[ \sum_{\substack{\sigma\in BNC(\chi)\\\pi\leq\sigma\leq\epsilon}}\mu_{BNC}(\pi, \sigma) \right] \mathcal{E}_{\pi}(Z_1,\ldots, Z_q)  \otimes F_\chi((i_1, \ldots, i_q), (j_1, \ldots, j_q))  \\
&= \sum_{\pi \in BNC(\chi)} \left[ \sum_{\substack{\sigma\in BNC(\chi)\\\pi\leq\sigma\leq\epsilon}}\mu_{BNC}(\pi, \sigma) \right] \E_{\pi}(Z_1\otimes F_{i_1, j_1}, \ldots, Z_q \otimes F_{i_q, j_q}).  \qedhere
\end{align*}
\end{proof}

\section{Partial Multivariate $R$-Transforms}
\label{sec:Transforms}

This section uses the combinatorial approach to bi-free probability to develop some partial $R$-transforms for pair of operators in the scalar setting.  Note all power series in this section are power series in commuting variables.

\subsection{Single Variable $\boldsymbol{R}$-Transforms}
\label{sec:Transform-preliminaries}

We begin by recalling some notation and standard results.  
\begin{defn}
Let $(\A, \varphi)$ be a non-commutative probability space and let $T \in \A$ be an arbitrary element.  For $n \in \mathbb{N}$, let $\kappa_n(T)$ denote the $n^{\text{th}}$ free cumulant of $T$; that is, in the notation of $(\ell, r)$-cumulants, $\kappa_n(T) = \kappa_{1_\chi}(T, T, \ldots, T)$ where $\chi : \{1,\ldots, n\} \to \{\ell, r\}$ is constant.

The \emph{Cauchy transform of $T$} is the power series
\[
G_T(z) = \varphi((z 1_\A - T)^{-1}) = \frac{1}{z} + \sum_{n\geq 1} \frac{\varphi(T^n)}{z^{n+1}},
\]
the \emph{$R$-transform of $T$} is the power series
\[
R_T(z) = \sum_{n\geq 0} \kappa_{n+1}(T) z^n,
\]
the \emph{moment series of $T$} is the power series
\[
M_T(z) = 1 + \sum_{n\geq 1} \varphi(T^n) z^n,
\]
and the \emph{cumulant series of $T$} is the power series
\[
C_T(z) = 1 + \sum_{n\geq 1} \kappa_n(T) z^n.
\]
\end{defn}
We recall the following relations between the above series from \cite{S1994}:
\begin{align}
G_T(z) &= \frac{1}{z}M_T\left(\frac{1}{z}\right)\\
C_T(z) &= zR_T(z) + 1 \label{eq:relating-cumulant-and-r-transform}\\
M_T(z) &= C_T(zM_T(z))  \label{eqn:moment-series-via-composition-with-cumulant-series}\\
C_T(z) &=M_T\left(\frac{z}{C_T(z)}\right) \label{eqn:cumulant-series-via-composition-with-moment-series}
\end{align}
Note the traditional relation $G_T\left(R_T(z) + \frac{1}{z}\right) = z$ in \cite{V1985} follows from these relations.  In addition, if $T$ and $S$ are freely independent, then $R_{T+S}(z) = R_T(z) + R_S(z)$.

\subsection{Partial $\boldsymbol{R}$-Transform for a Pair of Variables}
\label{sec:partial-R-transform-of-Voi}

This section re-derives the partial $R$-transform from \cite{V2013-2}*{Theorem 2.4}  via combinatorics.  To state \cite{V2013-2}*{Theorem 2.4} and to proceed with our proof, we make the following notation.
\begin{nota}
\label{not:notation-for-voiculescu-result}
Let $(\A, \varphi)$ be a non-commutative probability space and let $T, S \in \A$ be arbitrary elements.  We view $T$ as a left element and $S$ as a right element.  For $n, m \in \mathbb{N}\cup \{0\}$ with $n+m \neq 0$, let
\[
\kappa_{n, m}(T, S) = \kappa_{1_{\chi_{n,m}}}(Z_1, \ldots, Z_{n + m})
\]
where $\chi_{n,m} : \{1,\ldots, n+m\} \to \{\ell, r\}$,
\[
\chi_{n,m}(k)  = \left\{
\begin{array}{ll}
\ell & \mbox{if } k \leq n  \\
r & \mbox{if } k > n
\end{array} \right. ,
\qqand
Z_k  = \left\{
\begin{array}{ll}
T & \mbox{if } k \leq n  \\
S & \mbox{if } k > n
\end{array} \right..
\]

The \emph{two-variable Green's function} is the power series
\[
G_{T, S}(z,w) = \varphi((z1_\A - T)^{-1} (w 1_\A - S)^{-1}) = \frac{1}{zw} + \sum_{\substack{n,m\geq 0 \\ n+m \geq 1}} \frac{\varphi(T^nS^m)}{z^{n+1} w^{m+1}}.
\]
Further, let
\[
R_{T,S}(z,w) = \sum_{\substack{n,m\geq 0 \\ n+m \geq 1}} \kappa_{n, m}(T, S) z^n w^m.
\]
\end{nota}

The function $R_{T,S}(z,w)$ plays a similar role as the single variable $R$-transform when it comes to additive convolution.  Indeed, by Theorem \ref{thm:bifree-classifying-theorem}, if $(T_1, S_1)$ and $(T_2, S_2)$ are bi-free two-faced pairs, then 
\[
R_{T_1+T_2, S_1+S_2}(z,w) = R_{T_1, S_1}(z,w) + R_{T_2, S_2}(z,w).
\]

Using the above notation, the following result was proved in \cite{V2013-2} using analytic techniques.
\begin{thm}[\cite{V2013-2}*{Theorem 2.4}]
\label{thm:Voi-left-right-separated-R-Formula}
If $(T, S)$ is a two-faced pair in a Banach algebra non-commutative probability space, then
\begin{align}
R_{T, S}(z,w) = 1 + zR_T(z) + wR_S(w) - \frac{zw}{G_{T, S}\left(R_T(z) + \frac{1}{z}, R_S(w) + \frac{1}{w}\right)} \label{eq:Voi-partial-R-transform}
\end{align}
as holomorphic functions near $(0, 0) \in \bC^2$.
\end{thm}

To use combinatorics to produce a proof of Theorem \ref{thm:Voi-left-right-separated-R-Formula}, we need the following analogues of the moment and cumulant series.

\begin{defn}
The \emph{left-then-right cumulant series of $(T, S)$} is the power series
\begin{align}
C_{T, S}(z, w) = 1 + \sum_{\substack{n,m\geq 0 \\ n+m \geq 1}} \kappa_{n, m}(T, S) z^n w^m = 1+ R_{T, S}(z,w). \label{eq:LR-cumulant-to-partial-R-transform}
\end{align}
The \emph{left-then-right moment series of $(T, S)$} is the power series
\[
M_{T, S}(z, w) = 1 + \sum_{\substack{n,m\geq 0 \\ n+m \geq 1}} \varphi(T^nS^m)z^n w^m.
\]
\end{defn}
It is easy to verify that
\begin{align}
G_{T, S}(z, w) = \frac{1}{zw} M_{T, S}\left(\frac{1}{z}, \frac{1}{w}\right). \label{eq:Green-to-LR-Moment-Series}
\end{align}
\begin{thm}
\label{thm:my-left-right-separated-R-Formula}
Let $(\A, \varphi)$ be a non-commutative probability space and let $T, S \in \A$ be arbitrary elements.  Then
\begin{align}
M_T(z) + M_S(w) = \frac{M_T(z)M_S(w)}{M_{T,S}(z,w)} + C_{T, S}(zM_T(z), wM_S(w)). \label{eq:my-left-right-separated-R-Formula}
\end{align}
\end{thm}
\begin{rem}
Note Theorem \ref{thm:Voi-left-right-separated-R-Formula} and Theorem \ref{thm:my-left-right-separated-R-Formula} are equivalent.  Indeed, replacing $z$ and $w$ with $\frac{z}{C_T(z)}$ and $\frac{w}{C_S(w)}$ respectively in equation (\ref{eq:my-left-right-separated-R-Formula}) produces
\[
C_T(z) + C_S(w) = \frac{C_T(z)C_S(w)}{M_{T,S}\left( \frac{z}{C_T(z)}, \frac{w}{C_S(w)}  \right)} + C_{T, S}(z,w)
\]
via equation (\ref{eqn:cumulant-series-via-composition-with-moment-series}).  Using equation (\ref{eq:Green-to-LR-Moment-Series}), 
\[
\frac{C_T(z)C_S(w)}{M_{T,S}\left( \frac{z}{C_T(z)}, \frac{w}{C_S(w)}  \right)} = \frac{zw}{\frac{z}{C_T(z)}\frac{w}{C_S(w)}M_{T,S}\left( \frac{z}{C_T(z)}, \frac{w}{C_S(w)}\right)} = \frac{zw}{G_{T,S}\left(\frac{C_T(z)}{z}, \frac{C_S(w)}{w} \right)}.
\]
Therefore, by equations (\ref{eq:relating-cumulant-and-r-transform}) and (\ref{eq:LR-cumulant-to-partial-R-transform}),
\begin{align*}
R_{T, S}(z,w) &= C_{T,S}(z,w) - 1\\
&= C_T(z) + C_S(w) - \frac{zw}{G_{T,S}\left(\frac{C_T(z)}{z}, \frac{C_S(w)}{w} \right)} - 1\\
&= (zR_T(z) + 1) + (wR_S(w) + 1) - \frac{zw}{G_{T,S}\left( R_T(z) + \frac{1}{z}, R_S(w) + \frac{1}{w}\right)} - 1 \\
&= 1 + zR_T(z) + wR_S(w) - \frac{zw}{G_{T, S}\left(R_T(z) + \frac{1}{z}, R_S(w) + \frac{1}{w}\right)}
\end{align*}
which is equation (\ref{eq:Voi-partial-R-transform}).
\end{rem}
\begin{rem}
If $T$ and $S$ are such that the pairs $(T, 1_\A)$ and $(1_\A, S)$ are bi-free, it is easy to see that Theorem \ref{thm:my-left-right-separated-R-Formula} holds.  Indeed, under these assumptions, 
\[
M_{T,S}(z,w) = M_T(z) M_S(w) \qqand C_{T,S}(z,w) = C_T(z) + C_S(w) - 1.
\]
Thus, by equation (\ref{eqn:moment-series-via-composition-with-cumulant-series}),
\begin{align*}
\frac{M_T(z)M_S(w)}{M_{T,S}(z,w)} + C_{T, S}(zM_T(z), wM_S(w)) &= 1 + (C_T(zM_T(z)) + C_S(wM_S(w)) - 1)\\
&= M_T(z) + M_S(z).
\end{align*}
\end{rem}
\begin{proof}[Proof of Theorem \ref{thm:my-left-right-separated-R-Formula}]
For $n, m \in \mathbb{N}\cup \{0\}$ with $n+m \neq 0$, using the notation in \ref{not:notation-for-voiculescu-result} notice
\begin{align*}
\varphi(T^nS^m) &= \sum_{\pi \in \bncs(\chi_{n,m})} \kappa_\pi(Z_1, \ldots, Z_{n+m}) + \sum_{\substack{\pi \in BNC(\chi_{n,m}) \\ \pi \notin \bncs(\chi_{n,m})}} \kappa_\pi(Z_1, \ldots, Z_{n+m}) \\
&= \varphi(T^n)\varphi(S^m) + \sum_{\substack{\pi \in BNC(\chi_{n,m}) \\ \pi \notin \bncs(\chi_{n,m})}} \kappa_\pi(Z_1, \ldots, Z_{n+m}).
\end{align*}
Let 
\[
\Theta_{n,m} = \sum_{\substack{\pi \in BNC(\chi_{n,m}) \\ \pi \notin \bncs(\chi_{n,m})}} \kappa_\pi(Z_1, \ldots, Z_{n+m}).
\]
Notice $\Theta_{0, m} = \Theta_{n,0} = 0$ for all $n,m \neq 0$.  Otherwise, for $n,m \in \mathbb{N}$, note every partition $\pi \in BNC(\chi_{n,m}) \setminus \bncs(\chi_{n,m})$ must have a block $W$ such that 
\[
W \cap \{1, \ldots, n\} \neq \emptyset \qqand W \cap \{n+1, \ldots, n+m\} \neq \emptyset
\]
(that is, $\pi$ has a block with both left and right indices).  Let $V_\pi$ denote the block of $\pi$ with both left and right indices such that $\min(V_\pi)$ is smallest among all blocks $W$ of $\pi$ with both left and right indices.  

Rearrange the sum in $\Theta_{n,m}$ by first choosing $t \in \{1,\ldots, n\}$, $s \in \{1,\ldots, m\}$, and $V \subseteq \{1,\ldots, n+m\}$ such that
\begin{align*}
V_\ell &:= V \cap \{1, \ldots, n\}=\{u_1 < u_2 < \cdots < u_t\} \qand\\
V_r &:= V \cap \{n+1, \ldots, n+m\} = \{v_1 < v_2 < \cdots < v_s\},
\end{align*}
and then sum over all $\pi \in BNC(\chi_{n,m}) \setminus \bncs(\chi_{n,m})$ such that $V_\pi = V$.  If one defines $u_0 = 0$, $v_0=n$, $u_{t+1} = n+1$, and $v_{s+1} = n+m+1$, the fact that $\pi \in BNC(\chi_{n,m}) \setminus \bncs(\chi_{n,m})$ implies if $V_\pi = V$ then no block of $\pi$ contains indices from both intervals $(u_{k_1}, u_{k_1+1})$ and $(u_{k_2}, u_{k_2+1})$ when $k_1 \neq k_2$, from both intervals $(v_{k_1}, v_{k_1+1})$ and $(v_{k_2}, v_{k_2+1})$ when $k_1 \neq k_2$, and from both intervals $(u_{k_1}, u_{k_1+1})$ and $(v_{k_2}, v_{k_2+1})$ unless $k_1 = t$ and $k_2 = s$.  In particular, examining all $\pi$ such that $V_\pi = V$, each $(t+s+1)$-tuple consisting of bi-non-crossing partitions on each of the sets $(u_k, u_{k+1})$ for $k \in \{0, 1, \ldots, t-1\}$, $(v_k, v_{k+1})$ for $k \in \{0,1, \ldots, s-1\}$, and $(u_t, u_{t+1}) \cup (v_s, v_{s+1})$ occurs precisely once.  Since rearranging the sum gives
\begin{align*}
\Theta_{n,m} &= \sum^n_{t=1} \sum^m_{s=1} \sum_{\substack{V \\V_\ell = \{u_1 < u_2 < \cdots < u_t\} \\
V_r = \{v_1 < v_2 < \cdots < v_s\}   }}\sum_{\substack{\pi \in BNC(\chi_{n,m}) \\ \pi \notin \bncs(\chi_{n,m}) \\ V_\pi = V} } \kappa_\pi(Z_1, \ldots, Z_{n+m})
\end{align*}
and since the right most sum equals
\[
\kappa_{t,s}(T, S) \varphi(T^{n-u_t}S^{m-v_s})\prod^t_{k=1}\varphi(T^{u_k - u_{k-1} - 1}) \prod^s_{k=1} \varphi(S^{v_k - v_{k-1} - 1}),
\]
we obtain $\Theta_{n,m}$ equals
\[
\sum_{\substack{t \in \{1,\ldots, n\} \\i_0, i_1, \ldots, i_t \in \{0, 1, \ldots, n\} \\ i_0 + i_1 + \cdots + i_t = n-t}}  \sum_{\substack{s \in \{1,\ldots, m\} \\j_0, j_1,\ldots, j_s \in \{0, 1, \ldots, m\} \\ j_0 + j_1 + \cdots + j_s = m-s}} \kappa_{t,s}(T, S) \varphi(T^{i_0}S^{j_0}) \prod^t_{k=1}  \varphi(T^{i_k}) \prod^s_{k=1} \varphi(S^{j_k}).
\]

Using the above, notice
\begin{align*}
M_{T,S}(z,w) &= 1 + \sum_{n\geq 1} \varphi(T^n)z^n + \sum_{m\geq 1} \varphi(S^m)w^m + \sum_{n,m\geq 1} \varphi(T^nS^m) z^n w^m \\
&= M_T(z)M_S(w) + \sum_{n,m\geq 1} \Theta_{n,m} z^n w^m.
\end{align*}
By rearranging the remaining sum on the right to collect each $\kappa_{t,s}(T,S)$ term, we obtain
\begin{align*}
\sum_{n,m\geq 1}& \Theta_{n,m} z^n w^m \\
&= \sum_{s,t\geq 1} \kappa_{t,s}(T,S) M_{T,S}(z,w)  M_T(z)^t M_S(w)^s  z^t w^s\\
&= \left(\sum_{s,t\geq 1} \kappa_{t,s}(T,S)   (zM_T(z))^t (sM_S(w))^s\right)M_{T,S}(z,w) \\
&= \left(C_{T,S}(zM_T(z), sM_S(w)) - C_T(zM_T(z)) - C_S(wM_S(w)) + 1\right)M_{T,S}(z,w) \\
&= \left(C_{T,S}(zM_T(z), sM_S(w)) - M_T(z) - M_S(w) + 1\right)M_{T,S}(z,w).
\end{align*}
By combining these equations and applying simple algebraic manipulations, the result follows.
\end{proof}

\subsection{Partial $\boldsymbol{R}$-Transform for Bi-Free Boolean Systems}

This section develops an additional partial $R$-transform in a specific context which provides a more general result than that of \cite{SW1997}*{Proposition 2.1} for Boolean independent algebras.  

To produce such a partial $R$-transform, the technique from Section \ref{sec:partial-R-transform-of-Voi} are utilized along with some simple combinatorics.  We begin with the following notation and definitions.
\begin{defn}
Let $\chi : \{1,\ldots, n\} \to \{\ell, r\}$ and let $m \in \mathbb{N}\cup \{0\}$.  For $\theta \in \{\ell, r\}$, we say that \emph{$\chi$ starts with $\theta$ precisely $m$ times} if 
\begin{enumerate}
\item $m=n$ and $\chi(k) = \theta$ for all $k$, or
\item $m < n$, $\chi(k) = \theta$ for all $1 \leq k \leq m$, and $\chi(m+1) \neq \theta$.  
\end{enumerate} 
Similarly, we say that \emph{$\chi$ ends with $\theta$ precisely $m$ times} if 
\begin{enumerate}
\item $m=n$ and $\chi(k) = \theta$ for all $k$, or
\item $m < n$, $\chi(k) = \theta$ for all $n-m < k \leq n$, and $\chi(n-m) \neq \theta$.  
\end{enumerate} 
Let $\S_{m,\theta}$ denote all $\chi : \{1,\ldots, n\} \to \{\ell, r\}$ over all $n \in \mathbb{N}$ such that $\chi$ starts with $\theta$ precisely $m$ times and let  $\E_{m,\theta}$ denote all $\chi : \{1,\ldots, n\} \to \{\ell, r\}$ over all $n \in \mathbb{N}$ such that $\chi$ ends with $\theta$ precisely $m$ times.
\end{defn}
\begin{defn}
\label{defn:notation-for-boolean-RTransform}
Let $(\A, \varphi)$ be a non-commutative probability space and let $T, S \in \A$ be arbitrary elements.  We view $T$ as a left element and $S$ as a right element.  For $\chi : \{1,\ldots, n\} \to \{\ell, r\}$, we define
\[
\varphi_\chi(T, S) = \varphi(Z_1 \cdots Z_n) \qqand \kappa_{\chi}(T, S) = \kappa_{\chi}(Z_1, \ldots, Z_{n})
\]
where 
\[
Z_k  = \left\{
\begin{array}{ll}
T & \mbox{if } \chi(k) = \ell \\
S & \mbox{if } \chi(k) = r
\end{array} \right. .
\]
Define $L_\chi = |\{k \, \mid \, \chi(k) =\ell\}|$  and $R_\chi = |\{k \, \mid \, \chi(k) = r\}|$.

The \emph{commutative moment series of $(T, S)$} is the power series
\[
M^c_{T, S}(z,w) = 1 + \sum_{n\geq 1} \sum_{\chi : \{1, \ldots, n\} \to \{\ell, r\}} \varphi_\chi(T, S) z^{L_\chi}w^{R_\chi}.
\]

For $m \in \mathbb{N} \cup \{0\}$ and $\theta \in \{\ell, r\}$, the \emph{$\theta$-starting of length $m$ moment series of $(T, S)$} is the power series
\[
M_{T,S}^{m, \theta}(z,w) = \sum_{\chi \in \S_{m,\theta}} \varphi_\chi(T, S) z^{L_\chi}w^{R_\chi}.
\]
Similarly,  the \emph{$\theta$-ending of length $m$ cumulant series of $(T, S)$} is the power series
\[
C_{T,S}^{m, \theta}(z,w) = \sum_{\chi \in \E_{m,\theta}} \kappa_\chi(T, S)z^{L_\chi}w^{R_\chi}.
\]
\end{defn}

\begin{thm}
\label{thm:my-mixed-R-Formula}
Let $(\A, \varphi)$ be a non-commutative probability space and let $T, S \in \A$ be arbitrary elements such that $\varphi(T^n) = \varphi(S^n) = 0$ for all $n\geq 1$.  Then
\begin{align}
\label{eq:my-combined-R-Formula} M^c_{T,S}(z,w) = 1 &+ \sum_{m\geq 1} C^{m,\ell}_{T,S}(z,w)\left(1 + \sum_{k\geq 0} {m+k \choose k} M^{k,r}_{T,S}(z,w)   \right)  \\
& +\sum_{m\geq 1} C^{m,r}_{T,S}(z,w) \left(1 + \sum_{k\geq 0}  {m+k \choose k} M^{k,\ell}_{T,S}(z,w)    \right).  \notag
\end{align}
\end{thm}
\begin{rem}
Both Theorem \ref{thm:my-left-right-separated-R-Formula} and Theorem \ref{thm:my-mixed-R-Formula} have embedded conditions which reduce the combinatorics of bi-non-crossing partitions.  Indeed Theorem \ref{thm:my-left-right-separated-R-Formula} requires all left operators occur before all right operators which enables a collection of $\kappa_{n,m}(T, S)$ terms.  Similarly, the proof of Theorem \ref{thm:my-mixed-R-Formula} uses the condition $\varphi(T^n) = \varphi(S^n) = 0$ to determine the coefficient of each $\kappa_{\chi}(T, S)$.
\end{rem}
\begin{proof}[Proof of Theorem \ref{thm:my-mixed-R-Formula}]
First, consider a fixed $\chi : \{1,\ldots, n\} \to \{\ell, r\}$.  Using the notation in Definition \ref{defn:notation-for-boolean-RTransform}, notice
\begin{align*}
\varphi_\chi(T, S) &= \sum_{\pi \in \bncs(\chi)} \kappa_\pi(Z_1, \ldots, Z_{n}) + \sum_{\substack{\pi \in BNC(\chi) \\ \pi \notin \bncs(\chi)}} \kappa_\pi(Z_1, \ldots, Z_{n}) \\
&= \varphi(T^{L_\chi})\varphi(S^{R_\chi}) + \sum_{\substack{\pi \in BNC(\chi) \\ \pi \notin \bncs(\chi)}} \kappa_\pi(Z_1, \ldots, Z_{n})\\
&=  \sum_{\substack{\pi \in BNC(\chi) \\ \pi \notin \bncs(\chi)}} \kappa_\pi(Z_1, \ldots, Z_{n}).
\end{align*}
Notice $\varphi_\chi(T, S) = 0$ if $L_\chi = 0$ or $R_\chi = 0$.  Otherwise, by repeating arguments similar to those used in Theorem \ref{thm:my-left-right-separated-R-Formula} (i.e. summing over all $\pi$ with the same first block $V_\pi$ with both left and right nodes) along with the fact that $\varphi(T^k) = \varphi(S^k) = 0$ for all $k\geq 1$, we see that if
\[
\chi^{-1}(\{\ell\}) = \{ i_1 < i_2 < \cdots < i_{L_\chi} \} \qqand \chi^{-1}(\{r\}) = \{j_1 < j_2 < \cdots < j_{R_\chi}\},
\]
then
\[
\varphi_\chi(T, S) = \sum^{L_\chi}_{t=1} \sum^{R_\chi}_{s=1} \kappa_{\chi|_{V^\chi_{t,s}}}(T, S) \varphi_{\chi|_{\left(V^\chi_{t,s}\right)^c}}(T, S)
\]
where
\[
V^\chi_{t,s} = \{i_1, i_2, \ldots, i_t\} \cup \{j_1, j_2, \ldots, j_s\}.
\]

Using the above, and the fact that $\varphi(T^k) = \varphi(S^k) = 0$ for all $k \geq 1$, we see that
\begin{align*}
M^c_{T,S}(z,w) = 1 + \sum_{n\geq 1} \sum_{\chi : \{1,\ldots, n\} \to \{\ell, r\}} \sum^{L_\chi}_{t=1} \sum^{R_\chi}_{s=1} \kappa_{\chi|_{V^\chi_{t,s}}}(T, S) \varphi_{\chi|_{\left(V^\chi_{t,s}\right)^c}}(T, S)z^{L_\chi} w^{R_\chi}.
\end{align*}
Equation (\ref{eq:my-combined-R-Formula}) then follows from the above equation since if $\theta_1, \theta_2 \in \{\ell, r\}$ are such that $\theta_1 \neq \theta_2$, if $m_1, m_2 \in \mathbb{N} \cup \{0\}$ with $m_1 \neq 0$, if $\chi_1 \in \E_{m_1, \theta_1}$, and if $\chi_2 \in \S_{m_2, \theta_2}$, then the number of $\chi : \{1,\ldots, n\} \to \{\ell, r\}$ such that $\chi|_{V^\chi_{t,s}} = \chi_1$ and $\chi|_{\left(V^\chi_{t,s}\right)^c} = \chi_2$ for some $t, s$ is precisely ${m_1 + m_2 \choose m_2}$ (the number of ways to uniquely arrange last $m_1$ nodes of $\chi_1$ valued $\theta_1$ with the first $m_2$ nodes of $\chi_2$ valued $\theta_2$).
\end{proof}
\begin{rem}
Repeating the proof of Theorem \ref{thm:my-mixed-R-Formula}, one can easily see \cite{SW1997}*{Proposition 2.1}.  Indeed for $T = T_{Z}$ and $S = S_{1}$ as in Construction \ref{cons:Boolean}, we see that $\varphi_\chi(T, S) = 0$ unless $\chi$ is alternating.  Thus only alternating $\chi$ need be considered in the proof of Theorem \ref{thm:my-mixed-R-Formula}.  It is then clear that for
\[
\kappa_{\chi|_{V^\chi_{t,s}}}(T, S) \varphi_{\chi|_{\left(V^\chi_{t,s}\right)^c}}(T, S)
\]
to be non-zero, $t=s$ is required as $T^2 = S^2 = 0$ and $\varphi(STSTS \cdots) = 0$.  In this case, by Remark \ref{rem:boolean-cumulant-via-LR}, 
\[
\kappa_{\chi|_{V^\chi_{t,t}}}(T, S) \varphi_{\chi|_{\left(V^\chi_{t,t}\right)^c}}(T, S) = \kappa_{\text{B}, n}(Z) \varphi(Z^{n-t})
\]
where $\kappa_{\text{B}, n}(Z)$ is the $n^{\text{th}}$ Boolean cumulant of $Z$.  Therefore, by collecting terms, we see that
\begin{align*}
1 + \sum_{n\geq 1} \varphi(Z^n)(zw)^n &= 1 + \sum_{k\geq 1} \sum_{\substack{\chi : \{1,\ldots, 2k\} \to \{\ell, r\} \\ \chi \text{ alternating}}} \varphi_\chi(T, S) (zw)^k \\
&= 1 + \sum_{k\geq 1} \sum^k_{t=1} \kappa_{\text{B}, k}(Z) \varphi(Z^{k-t}) (zw)^k \\
&= 1 + \left(\sum_{k\geq 1}  \kappa_{\text{B}, k}(Z)(zw)^n\right) \left(1 + \sum_{n\geq 1} \varphi(Z^n)(zw)^n   \right)
\end{align*}
which is precisely the formula of \cite{SW1997}*{Proposition 2.1}.
\end{rem}

\section*{Acknowledgements}

The author would like to thank Dan Voiculescu, Dimitri Shlyakhtenko, and Serban Belinschi for posing questions which led to this paper.


\begin{thebibliography}{13}
\bib{BMS2013}{article}{
    author={Belinschi, S. T.},
    author={Mai, T.},
    author={Speicher, R.},
    title={Analytic subordination theory of operator-valued free additive convolution and the solution of a general random matrix problem},
    eprint={arXiv:1303.3196},
    pages={26}
}



\bib{B2005}{article}{
    author={Bercovici, H.},
    title={Multiplicative monotonic convolution},
	journal={Illinois J. Math.},
	volume={49},
	number={3},
    year={2005},
    pages={929-951}
}









\bib{CNS2014-1}{article}{
    author={Charlesworth, I.},
    author={Nelson, B.},
    author={Skoufranis, P.},
    title={On Two-Faced Families of Non-Commutative Random Variables},
	journal={to appear in Canad. J. Math.},
    year={2015},
    pages={26 pages}
}




\bib{CNS2014-2}{article}{
    author={Charlesworth, I.},
    author={Nelson, B.},
    author={Skoufranis, P.},
    title={Combinatorics of Bi-Free Probability with Amalgamation},
	journal={to appear in Comm. Math. Phys.},
    year={2014},
    pages={34 pages}
}





\bib{HS2011}{article}{
    author={Hasebe, T.},
    author={Saigo, H.},
    title={The monotone cumulants},
    journal={Ann. Inst. Henri, Poincar\'{e} Probab. Stat.},
    volume={47},
    number={4},
    year={2011},
    pages={1160-1170}
}





\bib{L2004}{article}{
    author={Lehner, F.},
    title={Cumulants in noncommutative probability theory I. Noncommutative exchangeability systems},
    journal={Math. Z.},
    volume={248},
    number={1},
    year={2004},
    pages={67-100}
}





\bib{MN2013}{article}{
    author={Mastnak, M.},
    author={Nica, A.},
    title={Double-ended queues and joint moments of left-right canonical operators on full Fock space},
    eprint={arXiv:1312.0269},
    year={2013},
    pages={28}
}



\bib{M2001}{article}{
    author={Muraki, N.},
    title={Monotonic independence, monotonic central limit theorem and monotonic law of small numbers},
    journal={Infin. Dimens. Anal. Quantum Probab. Relat. Top.},
    volume={4},
    number={1},
    year={2001},
    pages={39-58}
}

\bib{M2002}{article}{
    author={Muraki, N.},
    title={The five independences as quasi-universal products},
    journal={Infin. Dimens. Anal. Quantum Probab. Relat. Top.},
    volume={5},
    number={1},
    year={2002},
    pages={113-134}
}

\bib{M2003}{article}{
    author={Muraki, N.},
    title={The five independences as natural products},
    journal={Infin. Dimens. Anal. Quantum Probab. Relat. Top.},
    volume={6},
    number={3},
    year={2003},
    pages={337-371}
}



\bib{N1996}{article}{
    author={Nica, A.},
    title={$R$-Transforms of Free Joint Distributions and Non-crossing Partitions},
    journal={J. Funct. Anal. },
    volume={135},
    number={2},
    year={1996},
    pages={271-296}
}


\bib{NSS2002}{article}{
    author={Nica, A.},
    author={Shylakhtenko, D.},
    author={Speicher, R.},
    title={Operator-valued distributions: I. Characterizations of freeness},
    journal={Int. Math. Res. Notices},
    volume={29},
    year={2002},
    pages={1509-1538}
}




\bib{NS2006}{book}{
	author={Nica, A.},
	author={Speicher, R.},
	title={Lectures on the Combinatorics of Free Probability}, 
	series={London Mathematics Society Lecture Notes Series},
	volume={335},
	publisher={Cambridge University Press},
	date={2006}
}

\bib{P2008}{article}{
    author={Popa, M.},
    title={A combinatorial approach to monotonic independence over a C$^*$-algebra},
    journal={Pacific J. Math.},
    volume={237},
    number={2},
    year={2008},
    pages={299-325}
}


\bib{P2009}{article}{
    author={Popa, M.},
    title={A new proof for the multiplicative property of the boolean cumulants with applications to operator-valued case},
    journal={Colloq Math},
    volume={117},
    number={1},
    year={2009},
    pages={81-93}
}






\bib{S1994}{article}{
    author={Speicher, R.},
    title={Multiplicative functions on the lattice of non-crossing partitions and free convolution},
    journal={Math. Ann.},
    volume={298},
    number={1},
    year={1994},
    pages={611-628}
}


\bib{S1997}{book}{
	author={Speicher, R.},
	title={On universal products}, 
	series={in Free probability theory, Fields Inst. Commun.},
	volume={12},
	publisher={American Mathematics Society},
	place={Providence, RI, pp. 257-266},
	date={1997}	,
}



\bib{SW1997}{book}{
	author={Speicher, R.},
	author={Woroudi, R.},
	title={Boolean convolution}, 
	series={in Free probability theory, Fields Inst. Commun.},
	volume={12},
	publisher={American Mathematics Society},
	place={Providence, RI, pp. 267-279},
	date={1997}	
}

\bib{V1985}{book}{
	author={Voiculescu, D.},
	title={Symmetries of some reduced free product C$^*$-algebras}, 
	series={Operator Algebras and Their Connection with Topology and Ergodic Theory, Springer Lecture Notes in Mathematics},
	place={Springer Verlag},
	volume={1132},
	date={1985},
	pages={556-588}
}




\bib{V1986}{article}{
    author={Voiculescu, D.},
    title={Addition of certain non-commuting random variables},
    journal={J. Funct. Anal.},
    volume={66},
    number={3},
    year={1986},
    pages={323-346}
}




\bib{V2013-1}{article}{
    author={Voiculescu, D.},
    title={Free Probability for Pairs of Faces I},
    journal={Comm. Math. Phys.},
    year={2014},
    volume={332},
    pages={955-980}
}


\bib{V2013-2}{article}{
    author={Voiculescu, D.},
    title={Free Probability for Pairs of Faces II: 2-Variable Bi-Free Partial $R$-Transform and Systems with Rank $\leq 1$ Commutation},
    eprint={arXiv:1308.2035},
    year={2013},
    pages={21}
}




\end{thebibliography}
\end{document}